\documentclass{article}
\usepackage[utf8]{inputenc}
\usepackage{mathtools, nccmath}
\usepackage{geometry}
\usepackage{xcolor}
\usepackage{amsmath}
\usepackage{multirow}
\usepackage{amsthm}
\usepackage{changepage}
\usepackage{amssymb}
\usepackage{bbm}
\usepackage{xcolor}
\usepackage{graphicx}
\usepackage{algorithm,algorithmic}
\usepackage{cases}
\usepackage{enumitem}
\usepackage{subfig}
\usepackage[export]{adjustbox}
\usepackage[english]{babel}
\DeclareMathOperator*{\softmin}{soft-min}
\DeclareMathOperator*{\argmine}{argmin-e}
\newtheorem{theorem}{Theorem}[section]

\newtheorem{lemma}[theorem]{Lemma}
\newtheorem{proposition}[theorem]{Proposition}
\newtheorem{assumption}[theorem]{Assumption}
\newtheorem{remark}[theorem]{Remark}
\newtheorem{definition}[theorem]{Definition}

\DeclareMathOperator*{\argmin}{arg\,min}

\DeclareFontFamily{U}{mathx}{\hyphenchar\font45}
\DeclareFontShape{U}{mathx}{m}{n}{
      <5> <6> <7> <8> <9> <10>
      <10.95> <12> <14.4> <17.28> <20.74> <24.88>
      mathx10
      }{}
\DeclareSymbolFont{mathx}{U}{mathx}{m}{n}
\DeclareFontSubstitution{U}{mathx}{m}{n}
\DeclareMathAccent{\widecheck}{0}{mathx}{"71}
\DeclareMathAccent{\wideparen}{0}{mathx}{"75}

\usepackage[colorlinks,citecolor=blue]{hyperref}
\title{Convergence of Multiscale Reinforcement Q-Learning Algorithms for Mean Field Game and Control Problems
}
\author{
Andrea Angiuli\thanks{Prime Machine Learning Team, Amazon. 320 Westlake Ave N, SEA83, Seattle, WA, 98109 (E-mail: \href{mailto:aangiuli@amazon.com}{aangiuli@amazon.com}). The work presented here does not relate to this author's position at Amazon. }
  \and Jean-Pierre Fouque\thanks{Department of Statistics and Applied Probability, South Hall, University of California, Santa Barbara, CA 93106, USA (E-mail: \href{mailto:fouque@pstat.ucsb.edu}{fouque@pstat.ucsb.edu}). Work supported by NSF grants DMS-1814091 and DMS-1953035.} 
  \and Mathieu Lauri\`ere\thanks{Shanghai Frontiers Science Center of Artificial Intelligence and Deep Learning; NYU-ECNU Institute of Mathematical Sciences at NYU Shanghai; NYU Shanghai, 567 West Yangsi Road, Shanghai, 200126, People’s Republic of China (E-mail: 
  \href{mailto:mathieu.lauriere@nyu.edu}{mathieu.lauriere@nyu.edu}).}
  \and  Mengrui Zhang\thanks{Department of Statistics and Applied Probability, South Hall, University of California, Santa Barbara, CA 93106, USA (E-mail: \href{mailto:mengrui@umail.ucsb.edu}{mengrui@umail.ucsb.edu}).
  }}
\begin{document}

\maketitle
\centerline{\textbf{Abstract}}
\begin{adjustwidth}{50pt}{50pt} 
We establish the convergence of the unified two-timescale Reinforcement Learning (RL) algorithm presented in \cite{Andrea20}. This algorithm provides solutions to Mean Field Game (MFG) or Mean Field Control (MFC) problems depending on the ratio of two learning rates, one for the value function and the other for the mean field term. 
Our proof of convergence highlights the fact that in the case of MFC several mean field distributions need to be updated and for this reason we present two separate algorithms, one for MFG and one for MFC.
We focus on a setting with finite state and action spaces, discrete time and infinite horizon. The proofs of convergence rely on 
a generalization of the  two-timescale approach of \cite{Borkar97}. 
The accuracy of approximation to the true solutions depends on the smoothing of the policies. We  provide a numerical example illustrating the convergence. 
\end{adjustwidth}
\section{Introduction}

Reinforcement learning (RL) is a type of machine learning technique that enables an agent to learn in an interactive environment by trial and error using feedback from its own actions and experiences, which is formalized through the concept of Markov Decision Process (MDP). For a general introduction, we refer to~\cite{sutton2018reinforcement}. In the past decade, reinforcement learning has attracted a growing interest and led to several breakthroughs in various fields from classical games, such as Atari~\cite{mnih2015human} or Go~\cite{silver2016mastering},  to robotics (see e.g.~\cite{gu2017deep,vecerik2017leveraging}), and more recently to applications in training models to follow human feedback (see e.g.~\cite{ouyang2022training}). While classical RL focuses on a single agent, multi-agent RL (MARL) aims at extending the paradigm to situations in which multiple agents learn while interacting. We refer to e.g.~\cite{busoniu2008comprehensive,zhang2021multi} for more background, and to~\cite{lanctot2017unified,yang2020overview} for game-theoretic perspectives. Although great successes have been achieved, in particular in classical games, they remain mostly limited to situations with a small number of players. The scalability of model-free MARL methods in terms of the number of agents is a key challenge.

Concurrently, the past decade has witnessed the development of the theory of mean field games (MFGs), introduced by Lasry and Lions~\cite{MR2295621}, and Caines, Huang and Malhame~\cite{MR2346927}. MFGs provide a framework to study games with a very large number of anonymous players interacting in a symmetric manner. The theory has been extensively investigated. 
While the classical concept of MFG focuses on the notion of Nash equilibrium, it can also be relevant to consider the notions of social optimum, which gives rise to so-called mean field control (MFC) problems. We refer to~\cite{MR3134900,carmona2018probabilisticI-II} for more background on this topic.  

In the past few years, the question of learning solutions to MFGs and MFC problems using model-free RL methods has gained momentum. Various methods in different settings have been proposed in the literature; see~\cite{lauriere2022learning} for a survey. With the exception of~\cite{Andrea20}, which is the basis of the present paper, these methods focus on solving one of the two types of problems, MFG or MFC. 
On the one hand, to learn MFGs solutions, two classical families of methods are those relying on strict contraction and fixed point iterations (e.g., \cite{guo2019learning,cui2021approximately,anahtarci2023q} with tabular Q-learning or deep RL), and those relying on monotonicity and the structure of the game (e.g.,  \cite{elie2020convergence,perrin2020continuousfp,lauriere2022learning} using fictitious play and tabular or deep RL). Two-timescale analysis to learn MFG solutions has been used in \cite{mguni2018decentralised,SubramanianMahajan-2018-RLstatioMFG}. On the other hand, learning solutions to MFC  amounts to solving an MDP in which the state is the mean field. This type of problems has been called mean field Markov decision processes and has been studied e.g. in~\cite{carmona2023model,motte2022mean,gu2021meanQ,motte2023quantitative}. This leads to value-based RL methods such as Q-learning \cite{carmona2023model,gu2021meanQ}. Other methods include policy gradient~\cite{CarmonaLauriereTan-2019-LQMFRL}, actor-critic methods~\cite{frikha2023actor} or model-based RL methods~\cite{pasztor2021efficientmodelbased}.

Recently, \cite{Andrea20} proposed a common RL algorithm to solve MFGs and MFCs using a single algorithm by adjusting the learning rates. The algorithm iteratively updates an estimation of the population distribution and an estimation of the representative player's value function. Depending on the rates at which these updates are done, the same algorithm can converge to the MFG solution or the MFC solution. At a high level, the method combines the classical Q-learning updates~\cite{watkins1989learning} with updates of the mean field using a two-timescale scheme.

In this paper, we analyze multi-timescales algorithms that have been introduced and studied numerically  in~\cite{Andrea20} for mean field games and mean field control problems.
As in these references, the problems are studied in the context of finite state and action spaces, and in infinite time horizon. The key technical ingredient is the theory of stochastic approximation, see~\cite{Borkar97}, and in particular results pertaining to multiple timescales. Algorithms used in model-free RL are typically asynchronous and based on samples. To simplify the presentation, we start by analyzing the convergence of an idealized algorithm which uses synchronous and expectation-based updates. 
We then consider a version where the expected-based updates are replaced by sample-based updates. 
Last, we study a realistic algorithm, with asynchronous,  sample-based updates. 
Intuitively, each of these two steps corresponds to incorporating one approximation: when the number of samples is large enough, sample-based updates mimic expectation-based updates, and when every state-action pair is visited frequently enough in an asynchronous method, the algorithm's behavior is close to its synchronous counterparts. Our proof of convergence reveals that in the case of MFC, one needs to update multiple population distributions, one for each (state, action) point. For clarity, we write two separate algorithms, one for MFG and one for MFC, and we provide a proof of convergence for each. 

\paragraph{Related works. } As mentioned above, this paper proves the convergence of the algorithms proposed in~\cite{Andrea20}. The algorithms have been extended to the finite time horizon setting in\\
\cite{angiuli2023reinforcement} and to deep actor-critic methods in~\cite{angiuli2023deep}. 
Closely related to this family of methods, learning MFG solutions using RL and multiple timescales has been studied in\\
\cite{mguni2018decentralised,SubramanianMahajan-2018-RLstatioMFG}. These methods focus purely on he MFG setting, while our method presents a unified algorithm for both MFG and MFC. 
Besides the multi-timescale aspect, another feature of our work is the fact that the algorithms use the trajectory of a single representative player to learn the asymptotic mean field distribution. 
Along a similar direction, \cite{zaman2023oracle} proposed and analyzed an RL algorithm for MFGs in a setting where the representative agent does not have access to an oracle that can provide the mean field information. 
Last, the question of approximating the solution to mean field dynamics using the trajectory of a single particle has been studied  in~\cite{du2023sequential,du2023empirical,du2023self}, in the continuous space setting.

\paragraph{Structure of the paper. }

In Section~\ref{sec:notations} we introduce notations that will be used throughout the paper. In Sections~\ref{sec:MFG} and~\ref{sec:MFC} respectively, we present the main convergence results for mean field game (Theorem \ref{th:mainMFG}) and mean field control (Theorem \ref{th:mainMFC}) using the idealized algorithm (synchronous updates with expectations). In particular, Theorem~\ref{thMFGcvg} (resp. Theorem~\ref{thMFCcvg}) establishes the convergence  while Theorem~\ref{th:accuracyMFG} (resp. Theorem~\ref{mfc_main}) analyzes the accuracy in the mean field game setting (resp. mean field control setting). 
In Section~\ref{sec:example}, we present a simople example for which the assumptions are satisfied and we provide  numerical illustrations of the two-timescale time-scale algorithms. In Section~\ref{sec:sto-approx}, we extend the convergence analysis to the algorithm with sample-based updates, but still synchronous. Section~\ref{sec:sto-approx-asynchronous} further extends the analysis to the case of sample-based asynchronous updates. 
We conclude the paper in Section~\ref{sec:conclusion}.

\section{Presentation of the Model and the Algorithms}\label{sec:notations}
 
\paragraph{Notations. }
The following notations will be used for both mean field games and mean field control problems.

Let $\mathcal{{X}} =\{ x_0, \dots, x_{|\mathcal{{X}}|-1}\}$ be a finite state space and 
$\mathcal{{A}} =\{ a_0, \dots, a_{|\mathcal{{A}}|-1}\}$ be a finite action space. We denote by $\Delta^{|\mathcal{X}|}$ the simplex of probability measures on $\mathcal{X}$. We denote by $\boldsymbol{\delta}$ the indicator function, i.e., for any $x \in \mathcal{X}$, $\boldsymbol{\delta}(x) = [\mathbbm{1}_{x_0}(x),...,\mathbbm{1}_{x_{|\mathcal{{X}}|-1}}(x)]$, which is a vector full of $0$ except at the coordinate corresponding to $x$. 
Let $p$ : $\mathcal{X} \times \mathcal{X} \times \mathcal{A} \times \Delta^{|\mathcal{X}|}\to \Delta^{|\mathcal{X}|}$ be a transition kernel. We will sometimes view it as a function:
\begin{align*}
    p : \mathcal{X} \times \mathcal{X} \times \mathcal{A} \times \Delta^{|\mathcal{X}|}\to [0,1], \hspace{20pt}(x,x',a,\mu)\mapsto p(x'|x,a,\mu)
\end{align*}
which will be interpreted as the probability at any given time step to jump to state $x'$ starting from state $x$ and using action $a$, when the population distribution is $\mu$.

A policy $\pi\in \Pi$ is a collection of probability distributions $\{\pi(x)=\pi(\cdot|x), x\in \mathcal{X}\}$ on the set of actions $\mathcal{A}$.
We  denote by $\mathrm{P}^{{\pi},{\mu}}$ the transition kernel according to the distribution ${\mu}$ on $\mathcal{X}$ and the policy ${{\pi}}\in \Pi$,   defined for any  distribution $\tilde\mu \in \Delta^{|\mathcal{X}|}$ by:
\begin{equation}
\label{eq:def-Ptransitions}
    ( \tilde\mu\mathrm{P}^{\pi,\mu})(x) = \sum_{x'\in \mathcal{X}} \tilde\mu(x') \sum_{a\in \mathcal{A}} \pi(a|x') p(x|x',a,\mu), \qquad x \in \mathcal{X}.
\end{equation}

Let $f$: $\mathcal{X} \times \mathcal{A} \times \Delta^{|\mathcal{X}|} \to \mathbbm{R}$ be a running cost function. We interpret $f(x,a,\mu)$ as the one-step cost, at any given time step, incurred to a representative agent who is at state $x$ and uses action $a$ while the population distribution is $\mu$. We first give a quick review of classical Q-learning for one agent and without the effect of the population distribution $\mu$.

\paragraph{Classical Q-learning. }  Classical RL aims at solving a Markov Decision Process (MDP) using the following setting. At each discrete time $n$, the agent observes her state  $X_n$ and chooses an action $A_n$ based on it. Then the environment evolves and provides the agent with a new state $X_{n+1}$ and reports a reward  $r_{n+1}$. The goal of the agent is to find an optimal policy $\pi$  that  assigns to each state an optimal probability of actions in order to maximize the expected cumulative rewards. The problem can be recast as the problem of learning the optimal state-action value function, also called Q-function: $Q^\pi(x,a)$ represents the expected cumulative discounted rewards when starting at state $x$, using an action $a$, and then following policy $\pi$. Mathematically, 
$$
    Q^\pi(x,a) = \mathbbm{E}\left[\sum_{n=0}^{\infty}\gamma^n r_{n+1}|X_0 = x, A_0=a\right],
$$
where $r_{n+1} = r(X_n,A_n)$ is the instantaneous reward, $\gamma\in(0,1)$ is a discounting factor, $X_{n+1}$ is distributed according to a transition probability which depends on $X_n$ and the action drawn from $\pi(X_n)$. The goal is to compute the optimal Q-function defined as:
\[
    Q^*(x,a) = \max_{\pi} Q^\pi(x,a).
\]
To this end, the Q-learning method was introduced by \cite{watkins1989learning}. The basic idea is to iteratively sample an action $A_n \sim \pi(X_n)$ according to a behavior policy $\pi$, observe the induced state $X_{n+1}$ and reward $r_{n+1}$, and then update the Q-table according to the formula: 
\[
    Q(X_n,A_n)
        \gets Q(X_n,A_n) +  \rho\left[r_{n+1} + \gamma\max_{a'\in\mathcal{A}}Q(X_{n+1},a')-Q(X_n,A_n)\right]
\]
where $\rho \in (0,1]$ is a learning rate. 

Note that in our case, to be consistent with most of the MFG and MFC literature, {\bf we will minimize costs instead of maximizing rewards}.

The notions of mean field game equilibrium and mean field control optimum will be presented in the following sections. For now, let us present the RL algorithm introduced in~\cite{Andrea20}.

\paragraph{Algorithms. }
We reproduce \cite[Algorithm 1]{Andrea20} in Algorithm~\ref{algo:U2MFQL} for MFG and in Algorithm~\ref{algo:U2MFQL-MFC} for MFC. In \cite{Andrea20}, the algorithm was presented with several episodes to help the learning process. Here, to simplify the presentation, we use a single episode. Also, to make clear the differences between MFG and MFC we write two different algorithms. They differ from the ratio of the learning rates as well as from the fact that in the MFC case one keep track of a trajectory $(X_n^{(x,a)})$ as well as a distribution $(\mu_n^{(x,a)})$ for each $(x,a)$. At the beginning of the episode, an initial state $X_0$ is provided by the environment. The distribution according to which it is sampled does not matter if the total number of time steps $N_{steps}$ is large enough. Instead of a fixed number of steps, one could stop after a certain stopping criterion is achiever. One possibility is to stop when $ \lvert|{\mu_{n+1}-\mu_{n}}\rvert|_1 \leq tol_{\mu}$ and $\|Q_{n+1}-Q_n\|_{1,1}<tol_Q$, for some predefined tolerances $tol_{\mu}$ and $tol_Q$. 

The most important parameters of this algorithm are the learning rates. Our analysis will cover a wider range of possible parameters, but let us recall that in~\cite{Andrea20} the authors chose learning rates of the following specific form, inspired by the RL literature:
\begin{align*} %
    \rho^Q_{n,x,a}=\frac{1}{(1+\nu(x,a,n))^{\omega^Q}},\hspace{20pt}\rho^{\mu}_n=\frac{1}{(1+(n+1))^{\omega^{\mu}}}
\end{align*}
where $\omega^Q$ and $\omega^\mu$ are two positive constants, and $\nu(x,a,n)=\sum_{m=0}^{n}\mathbbm{1}_{\{(X_m,A_m)=(x,a)\}}$ is the number of times the process $(X_n,A_n)$ visits state $(x,a)$ up to time $n$. In fact, in this paper we consider learning rates of the population distribution which also depend on $\nu(x,a,n)$:
\begin{align*} %
    \rho^\mu_{n,x,a}=\frac{1}{(1+\nu(x,a,n))^{\omega^\mu}},
\end{align*}
allowing comparison of the learning rates at each point $(x,a)$ by only comparing the two numbers $\omega^Q$ and $\omega^\mu$. For the choice of an action, we use a $\softmin_\phi$ regularization (see \eqref{softmin} in Appendix \ref{app3} for its definition).

\begin{algorithm}[H]
\caption{Two-timescale Mean Field Game Q-learning \label{algo:U2MFQL}} 
\begin{algorithmic}[1]
\REQUIRE 
$N_{steps}$: number of steps; $\phi$: parameter for $\softmin$ policy; $(\rho^Q_{n,x,a})_{n,x,a}$: learning rates for the value function; $(\rho^\mu_{n,x,a})_{n,x,a}$: learning rates for the population distribution: MFG: $\rho^Q_{n,X_n,A_n}>\rho^{\mu}_{n,X_n,A_n}$
\STATE \textbf{Initialization}: $Q_{0}(x,a) =0$ and $\mu_{0}=[\frac{1}{|\mathcal{{X}}|},...,\frac{1}{|\mathcal{{X}}|}]$ for all $(x,a)\in \mathcal{{X}}\times \mathcal{{A}}$.

\STATE \textbf{Observe} $X_{0}\sim \mu_0$ \\

\STATE
{\textbf{Update mean field: } 
$\mu_{0} = \boldsymbol{\delta}(X_{0})$}
\\

\FOR{$n=0,1,2,\dots, N_{steps}-1$}

\STATE \textbf{Choose action } $A_{n} \sim \softmin_\phi Q_n(X_{n},\cdot)$ and observe the new state 
\[
X_{n+1} \sim p(\cdot|X_{n}, A_{n}, \mu_{n})
\]
provided by the environment %
\\

\STATE\textbf{Update population distributions: } 
\[
\mu_{n+1} = \mu_{n} + \rho^{\mu}_{n,X_n,A_n} (\boldsymbol{\delta}(X_{n+1}) - \mu_{n})
\]

\STATE \textbf{Observe } the cost $f_{n+1}=f\left(X_{n},A_{n},\mu_{n}\right)$

\STATE
\textbf{Update value function: } $Q_{n+1}(x,a) = Q_{n}(x,a)$ for all $(x,a) \neq (X_n,A_n)$, and 
    \[
    Q_{n+1}(X_n,A_n)
        = Q_n(X_n,A_n) +  \rho^Q_{n,X_n,A_n}\left[f_{n+1} + \gamma\min_{a'\in\mathcal{A}}Q_n(X_{n+1},a')-Q_n(X_n,A_n)\right]
    \]

\ENDFOR
\STATE {\textbf{Return}} $(\mu_{N_{steps}},Q_{N_{steps}})$
\end{algorithmic}
\end{algorithm}

\begin{algorithm}[H]
\caption{Two-timescale Mean Field Control Q-learning \label{algo:U2MFQL-MFC}} 
\begin{algorithmic}[1]
\REQUIRE 
$N_{steps}$: number of steps; $\phi$: parameter for $\softmin$ policy; $(\rho^Q_{n,x,a})_{n,x,a}$: learning rates for the value function; $(\rho^\mu_{n,x,a})_{n,x,a}$: learning rates for the population distribution: MFC: $\rho^Q_{n,X_n,A_n}<\rho^{\mu}_{n,X_n,A_n}$
\STATE \textbf{Initialization}: $Q_{0}(x,a) =0$ and $\mu_0^{(x,a)}=[\frac{1}{|\mathcal{{X}}|},...,\frac{1}{|\mathcal{{X}}|}]$ for all $(x,a)\in \mathcal{{X}}\times \mathcal{{A}}$.

\STATE \textbf{Observe} $X^{(x,a)}_{0}\sim \mu^{(x,a)}_0$ for all $(x,a)$\\

\STATE
{\textbf{Update mean fields: } 
$\mu^{(x,a)}_{0} = \boldsymbol{\delta}(X^{(x,a)}_{0})$}
\\

\FOR{$n=0,1,2,\dots, N_{steps}-1$}

\STATE \textbf{Choose action: } Choose $A^{(x,a)}_n = a$ if $X_n^{(x,a)} =x$, otherwise $A^{(x,a)}_{n} \sim \softmin_\phi Q_n(X^{(x,a)}_{n},\cdot)$.

\STATE \textbf{Observe} the new states 
\[
X^{(x,a)}_{n+1} \sim p(\cdot|X^{(x,a)}_{n}, A^{(x,a)}_{n}, \mu^{(x,a)}_{n})
\]

provided by the environment %
\\

\STATE\textbf{Update population distributions: } \[
    \mu^{(x,a)}_{n+1}
        = \mu^{(x,a)}_n +  \rho^\mu_{n,X^{(x,a)}_n,A^{(x,a)}_{n}}(\boldsymbol{\delta}(X^{(x,a)}_{n+1}) - \mu^{(x,a)}_{n}).
    \]

\STATE \textbf{Observe } the costs $f^{(x,a)}_{n+1}=f\left(X^{(x,a)}_{n},A^{(x,a)}_{n},\mu^{(x,a)}_{n}\right)$

\STATE
\textbf{Update value function: } If $X^{(x,a)}_n = x$,
    \[
    Q_{n+1}(x,a)
        = Q_n(x,a) +  \rho^Q_{n,X^{(x,a)}_n,A^{(x,a)}_n}\left[f^{(x,a)}_{n+1} + \gamma\min_{a'\in\mathcal{A}}Q_n(X^{(x,a)}_{n+1},a')-Q_n(x,a)\right]
    \]
    Otherwise, $Q_{n+1} = Q_n $

\ENDFOR
\STATE {\textbf{Return}} $(\mu^{(x,a)}_{N_{steps}},Q_{N_{steps}})$
\end{algorithmic}
\end{algorithm}

Note that the algorithms use the cost function $f$ and the transition distribution $p$ only through evaluations and samples respectively. In other words, the algorithms are \emph{model-free} from the point of view of the representative agent. Since it only uses samples and not expected values, such updates are sometimes referred to as \emph{sample-based updates}. Furthermore, at step $n$, the value function $Q_{n+1}$ differs from $Q_{n}$ at only one state-action pair, so the algorithm is \emph{asynchronous}. 

{\bf Note also that the algorithm differentiates the two cases, MFG and MFC, by the choice of the learning rate parameters $\omega^Q$ and $\omega^\mu$, but also  by the way the cost depends on the population distributions, a single one $\mu_n$ for MFG, and one, $\mu_n^{(x,a)}$, for each $(x,a)$  in the MFC case.}

For the analysis, we will start with an idealized version of the algorithm. We  start by considering a \emph{synchronous} algorithm with \emph{expected updates} in which, at each iteration, every state-action pair is updated using the expectation of the $Q$-learning updates. We will then consider a version with synchronous but sample-based updates. Finally, we will come back to the algorithms described in Algorithm~\ref{algo:U2MFQL} and Algorithm~\ref{algo:U2MFQL-MFC}
which can be viewed as  combinations of classical Q-learning with an estimation of the population distributions.

\section{Mean Field Game}\label{sec:MFG}

\subsection{Mean Field Game Formulation}
In the context of Asymptotic MFG introduced in \cite[Section 2.2]{Andrea20}, we can view the problem faced by an infinitesimal agent among the crowd as an MDP parameterized by the population distribution. Hence, given a population distribution $\mu$, the Q-function for a single agent can be defined as follows:
\begin{align}
    Q^{\pi}_\mu(x,a) := f(x,a,\mu)+\mathbbm{E}\left[\sum_{n\geq 1}\gamma^n f(X_n^{\pi,\mu}, A_n,\mu)\right],
\end{align}
where the expectation is over the randomness of $X^{\pi,\mu}$ and $A$ such that,
\begin{align*}
    \begin{cases}
    X_0^{\pi,\mu} = x, A_0 = a,\\
    X_{n+1}^{\pi,\mu} \sim p(\cdot|X_n^{\pi,\mu}, A_n, \mu),  \qquad n \ge 0\\
    A_n \sim \pi(\cdot|X_n^{\pi,\mu}), \qquad n \ge 0
    \end{cases}
\end{align*}
where $\pi:\mathcal{X}\to \Delta^{|\mathcal{A}|}$ is a policy taking the state $x$ as input and outputting a probability distribution over the action set. 
We denote by $\Pi$ the set of all such policies.

The optimal Q-function  is defined as:
\begin{align}
    Q^*_\mu(x,a) := \inf_{\pi \in \Pi} Q^{\pi}_{\mu}(x,a).
\end{align} 
Since $\mu$ is fixed, Equation (3.20) in \cite{sutton2018reinforcement} can be applied, so that it satisfies the Bellman equation:
\begin{align}\label{Bellman_MFG}
    Q^*_\mu (x,a) = f(x,a,\mu)+\gamma \sum_{x'\in\mathcal{X}}p(x'|x,a,\mu)\min_{a'}Q^*_\mu(x',a'). 
\end{align}
In short, we denote this Bellman equation (\ref{Bellman_MFG}) by $Q^*_\mu = \mathcal{B}_\mu Q^*_\mu$, where for every $\mu$, $\mathcal{B}_\mu$ is defined as: for every $Q: \mathcal{X} \times \mathcal{A} \to \mathbb{R}$,
$$
    \mathcal{B}_\mu Q := f(x,a,\mu)+\gamma \sum_{x'\in\mathcal{X}}p(x'|x,a,\mu)\min_{a'}Q(x',a').
$$ To be more precise, $Q^*_\mu$ is the fixed point of the Bellman operator $\mathcal{B}_\mu$. The Bellman equation~\eqref{Bellman_MFG} implies that there is at least one optimal policy $\pi^*$, i.e., such that:
$$
    Q^{\pi^*}_\mu(x,a) = Q^{*}_\mu(x,a), \qquad \hbox{ for all $(x,a)$.} 
$$
Indeed, we can take for $\pi^*$ any policy such that the support of $\pi^*(\cdot|x)$ is included in the set of actions that minimize $Q^*_\mu(x,\cdot)$.
However, there could be multiple optimal policies.\\

Now we introduce the notion of (mean field) Nash equilibrium. 
\begin{definition}
\label{def:MFGeq}
$(\hat{\mu},\hat{\pi})$ forms a Nash equilibrium if the following two statements hold:
\begin{enumerate}
    \item (best response) $Q^{\hat{\pi}}_{\hat{\mu}} = Q^*_{\hat{\mu}}$ %
    \item (consistency) $\hat{\mu} = \lim_{n\to\infty}\mathcal{L}(X_n^{\hat{\pi},\hat{\mu}})$.
\end{enumerate}
\end{definition}
To be more precise, the first statement means that $\hat{\pi}$ is optimal for a representative infinitesimal player when the population distribution is $\hat{\mu}$, and the second statement means that $\hat{\mu}$ is a fixed point of the function $\hat{\mu} \mapsto \mathrm{P}^{\hat{\pi},\hat{\mu}}\hat{\mu}$ where $\mathrm{P}^{\hat{\pi},\hat{\mu}}$ is the transition kernel according to the $\hat{\mu}$ and policy ${\hat{\pi}}$, see~\eqref{eq:def-Ptransitions}.

\subsection{Idealized Two-timescale Iteration Approach}\label{sec:MFGsimple}
In order to solve the problem we defined above (Definition~\ref{def:MFGeq}), let us first introduce an idealized deterministic two-timescale approach. We consider the following iterative procedure, where both variables are updated at each iteration but with different rates, denoted by $\rho^\mu_n \in \mathbbm{R}_+$ and $\rho^Q_n \in \mathbbm{R}_+$. Starting from an initial guess $(Q_0,\mu_0)\in\mathbbm{R}^{|\mathcal{X}|\times|\mathcal{A}|}\times \Delta^{|\mathcal{X}|}$,  define iteratively for $n = 0,1,\dots$:
\begin{align}\label{simpletwotimescaleapproach_mfg}
    \begin{cases}
        \mu_{n+1} = \mu_n + \rho_n^{\mu}\mathcal{P}_2(Q_n,\mu_n),\\
        Q_{n+1} = Q_n + \rho^Q_n\mathcal{T}_2(Q_n,\mu_n),
    \end{cases}
\end{align}
where\footnote{The $\softmin_\phi$ function and the dynamics depending on a distribution vector are defined in the Appendix.} the functions $\mathcal{P}_2: \mathbb{R}^{|\mathcal{X}|\times|\mathcal{A}|} \to \Delta^{|\mathcal{X}|}$ and $\mathcal{T}_2: \mathbb{R}^{|\mathcal{X}|\times|\mathcal{A}|} \to \mathbb{R}^{|\mathcal{X}| \times |\mathcal{A}|}$ are defined by:
\begin{align*}
    \begin{cases}
        \mathcal{P}_2(Q,\mu)(x) = \sum_{x'}\mu(x')p(x|x',\softmin_\phi Q(x'),\mu)- \mu(x),\\
        \mathcal{T}_2(Q,\mu)(x,a) = f(x,a,\mu) + \gamma\sum_{x'}p(x'|x,a,\mu)\min_{a'}Q(x',a')-Q(x,a).
    \end{cases}
\end{align*}
These two functions can also be expressed in the following way where $\mathrm{P}^{\pi,\mu}$ is defined in~\eqref{eq:def-Ptransitions}:
\begin{align*}
    \begin{cases}
    \mathcal{P}_2 (Q,\mu) = \mu\mathrm{P}^{\softmin_\phi Q, \mu} - \mu,\\
        \mathcal{T}_2 (Q,\mu) = \mathcal{B}_{\mu}Q - Q.
    \end{cases}
\end{align*}
Our proof requires the functions $\mathcal{T}_2$ and $\mathcal{P}_2$ to be Lipschitz continuous. In order to obtain this property, we use $\softmin_\phi$ instead of $\argmin$ under function $\mathcal{P}_2$ and we require Lipschitz continuity of $f$. By~\cite{Borkar97}, we will obtain that, if $\rho^\mu_n/\rho^Q_n\to 0$ as $n\to+\infty$, the solution $(\mu_n,Q_n)_n$ to the above iterations~\eqref{simpletwotimescaleapproach_mfg} has a behavior that is described (in a sense to be made precise later) by the following system of ODEs, in the regime where $\epsilon$ is small and $t$ goes to infinity:
\begin{align}\label{MFGODE}
\begin{cases}
    \dot{\mu}_t = \mathcal{P}_2(Q_t,\mu_t),\\
    \dot{Q}_t = \frac{1}{\epsilon}\mathcal{T}_2(Q_t,\mu_t),
\end{cases}
\end{align}
where $\rho^\mu_n/\rho^Q_n$ is thought of being of order $\epsilon \ll 1$.

\subsection{Convergence of the Idealized Two-timescale Iteration Scheme}
In this section, we will establish the convergence of the idealized two-timescale approach defined in \eqref{simpletwotimescaleapproach_mfg}.

We introduce the following assumption.
\begin{assumption}\label{fp_lipschitz}
    The cost function $f$ is bounded and is Lipschitz with respect to $\mu$, with Lipschitz constant denoted by $L_f$ when using the $L^1$ norm.
    The transition kernel $p(\cdot|\cdot,\cdot,\mu)$ is also Lipschitz with respect to $\mu$, with Lipschitz constant denoted by $L_p$ when using the $L^1$ norm. In other words, for every $x,a,\mu,\mu'$,
    \begin{align*}
        |f(x,a,\mu) - f(x,a,\mu')| 
        &\le L_f \|\mu - \mu'\|_1 = L_f \sum_x |\mu(x) - \mu'(x)|,
        \\ 
        \sum_{x'}|p(x'|x,a,\mu) - p(x'|x,a,\mu')| = \|p(\cdot|x,a,\mu)-p(\cdot|x,a,\mu')\|_1 
        &\le L_p \|\mu - \mu'\|_1 = L_p \sum_x |\mu(x) - \mu'(x)|.
    \end{align*}
\end{assumption}
Under Assumption \ref{fp_lipschitz}, we now show that the functions $\mathcal{P}_2$ and $\mathcal{T}_2$ are Lipschitz continuous.
\begin{proposition}\label{PLipschitz}
If Assumption~\ref{fp_lipschitz} holds, then
    the function $(Q,\mu) \mapsto \mathcal{P}_2(Q,\mu)$ is Lipschitz with respect to both $Q$ and $\mu$: 
\begin{align}\label{LipP2Q}
        \|\mathcal{P}_2(Q,\mu) - \mathcal{P}_2(Q',\mu)\|_{\infty}
\leq \phi|\mathcal{A}|\|Q-Q'\|_\infty,
    \end{align}
     and
     \begin{align}\label{LipP2mu}
        \|\mathcal{P}_2(Q,\mu) - \mathcal{P}_2(Q,\mu')\|_{\infty}
         \leq (L_p + 2-|\mathcal{X}|c_{min}^\phi)\|\mu-\mu'\|_1
        \leq (L_p + 2-|\mathcal{X}|c_{min})\|\mu-\mu'\|_1,
    \end{align}
     where
     \begin{align}
            c_{min}^\phi:= &\min_{x,x',Q,\mu} \mathrm{P}^{\softmin_\phi Q,\mu}(x',x) = \min_{x,x',Q,\mu}\sum_a {\softmin}_{\phi} Q(x',a)p(x|x',a,\mu)\nonumber\\
        &\geq \min_{x,x',a,\mu} p(x|x',a,\mu)\sum_a {\softmin}_{\phi} Q(x',a) =\min_{x,x',a,\mu} p(x|x',a,\mu)=: c_{min}\label{def:cmin}
        \end{align}
    \end{proposition}
\begin{proof}
First, we note that for two policies $\pi$ and $\pi'$, we have:
\begin{align}
        \|\mathrm{P}^{\pi, \mu}\mu - \mathrm{P}^{\pi', \mu}\mu\|_\infty
        &\leq\sum_{x}|\sum_{x'} \mu(x')\sum_a \pi(a|x')p(x|x',a,\mu)-\pi'(a|x') p(x|x',a,\mu))| \notag \\
        &\leq\sum_{x'} \mu(x')|\sum_a\sum_{x}(\pi(a|x')p(x|x',a,\mu)-\pi'(a|x')p(x|x',a,\mu))| \notag \\
        &\leq \sum_{x'} \mu(x') \sqrt{\sum_a 1^2}\|\pi(x')-\pi'(x')\|_2 \notag \\
        &\leq |\mathcal{A}|^{1/2} \sum_{x'} \mu(x') \|\pi(x')-\pi'(x')\|_2
        \notag 
        \\
        &\le |\mathcal{A}|^{1/2} \max_{x'} \|\pi(x')-\pi'(x')\|_2. 
        \label{diff-P-policies}
\end{align}  
Applying this with $\pi(x) = {\softmin}_\phi Q(x)$ and $\pi'(x) = {\softmin}_\phi Q'(x)$, we obtain:
\begin{align*}
        \|\mathcal{P}_2(Q,\mu) - \mathcal{P}_2(Q',\mu)\|_{\infty} &=\|\mu\mathrm{P}^{\softmin_\phi Q, \mu} - \mu\mathrm{P}^{\softmin_\phi Q', \mu}\|_\infty\\
        &\leq |\mathcal{A}|^{1/2} \max_{x'} \|{\softmin}_\phi Q(x')-{\softmin}_{\phi} Q'(x')\|_2\\
        &\leq|\mathcal{A}|^{\frac{1}{2}}\phi\max_{x'} \|Q(x')-Q'(x')\|_2\\
        &\leq|\mathcal{A}|\phi\max_{x'} \|Q(x')-Q'(x')\|_\infty\\
        &\leq \phi|\mathcal{A}|\|Q-Q'\|_\infty.
\end{align*}    
For $\mathcal{P}_2$, we first make the following remark: we can find $p(i,j)>\epsilon$, so that $1-N\epsilon>0$, and then define $q(i,j)=(p(i,j)-\epsilon)/(1-N\epsilon)$. Thus, $P = (1-N\epsilon)Q+\epsilon J$, where $J$ is the $N\times N$ matrix with all entries 1. We use this fact to calculate the total variation. Hence we have: 
\begin{align*}
        \|\mathcal{P}_2(Q,\mu) - \mathcal{P}_2(Q,\mu')\|_{\infty} &\leq \|\mu\mathrm{P}^{\softmin_\phi Q, \mu}-\mu'\mathrm{P}^{\softmin_\phi Q, \mu'}\|_\infty + \|\mu-\mu'\|_\infty \\
        &\leq \|\mu\mathrm{P}^{\softmin_\phi Q, \mu}-\mu'\mathrm{P}^{\softmin_\phi Q, \mu})\|_1 + \|\mu'\mathrm{P}^{\softmin_\phi Q, \mu}-\mu'\mathrm{P}^{\softmin_\phi Q, \mu'}\|_1+\|\mu-\mu'\|_1\\
        &\leq(1-|\mathcal{X}|c_{min}^\phi)\|\mu-\mu'\|_1 + L_p\|\mu-\mu'\|_1+\|\mu-\mu'\|_1\\
        &\leq (L_p + 2-|\mathcal{X}|c_{min}^\phi)\|\mu-\mu'\|_1\\
        &\leq (L_p + 2-|\mathcal{X}|c_{min})\|\mu-\mu'\|_1.
\end{align*}
\end{proof}
\begin{proposition}\label{TLipschitz}
    If Assumption~\ref{fp_lipschitz} holds,
    then the function $(Q,\mu) \mapsto \mathcal{T}_2(Q,\mu)$ is Lipschitz with respect to both $Q$ and $\mu$.
\end{proposition}
\begin{proof}
We note that:
    \begin{align*}
        \|\mathcal{T}_2(Q,\mu) - \mathcal{T}_2(Q',\mu)\|_{\infty} &\leq \|\mathcal{B}_\mu Q - \mathcal{B}_\mu Q'\|_\infty + \|Q-Q'\|_\infty\\
        &\leq \gamma \|\sum_{x'} p(x'|\cdot,\cdot,\mu)(\min_{a'} Q(x',a')-\min_{a'} Q'(x',a'))\|_\infty+\|Q-Q'\|_\infty\\
        &\leq \gamma \|Q-Q'\|_\infty + \|Q-Q'\|_\infty\\
        &\leq (\gamma + 1)\|Q-Q'\|_\infty,
    \end{align*}
    where we used the fact that $$\|\sum_{x'} p(x'|\cdot,\cdot,\mu)(\min_{a'} Q(x',a')-\min_{a'} Q'(x',a'))\|_\infty \le \sup_{x,a}  \sum_{x'} p(x'|x,a,\mu)\|(\min_{a'} Q(x',a')-\min_{a'} Q'(x',a'))\|_\infty,$$ with $p(x'|x,a,\mu) \le 1$ and $\|(\min_{a'} Q(x',a')-\min_{a'} Q'(x',a'))\|_\infty \le \|Q- Q'\|_\infty$. Hence:
    \begin{align*}
        \|\mathcal{T}_2(Q,\mu) - \mathcal{T}_2(Q,\mu')\|_{\infty} &\leq \|f(\cdot,\cdot,\mu) - f(\cdot,\cdot,\mu')\|_\infty + \|\gamma\sum_{x'}|p(x'|\cdot,\cdot,\mu)-p(x'|\cdot,\cdot,\mu')||\min_{a'} Q(x',a')|\|_\infty\\
        &\leq (L_f + \gamma L_p\|Q\|_\infty)\|\mu-\mu'\|_1\\
        &\leq |\mathcal{X}|(L_f + \gamma L_p\|Q\|_\infty)\|\mu-\mu'\|_\infty.
    \end{align*}
\end{proof}
Now that we have these Lipschitz properties, we show that the second ODE for $Q_t$ in the system \eqref{MFGODE} has a unique global asymptotically stable equilibrium (GASE).
\begin{proposition}\label{eQ_MFG}
    If Assumption \ref{fp_lipschitz} holds, then for any given $\mu$, the ODE $\dot{Q}_t = \mathcal{T}_2(Q_t,\mu)$ has a unique GASE, that we will denote by $Q^*_\mu$. Moreover,  $Q^*_\mu:\Delta^{|\mathcal{X}|}\to\mathbbm{R}^{|\mathcal{X}|\times|\mathcal{A}|}$ is Lipchitz.%
\end{proposition}
\begin{proof}We will use in the proof the fact that, by definition of the $Q$ table, $\|Q^*_{\mu}\|_\infty \leq \frac{1}{1-\gamma}\|f\|_\infty$. 
    By the procedure we used to prove Proposition \ref{TLipschitz}, we have $ \|\mathcal{B}_\mu Q - \mathcal{B}_\mu Q'\|_\infty \leq \gamma\|Q-Q'\|_\infty$, where $\gamma < 1$ implies that $Q \mapsto \mathcal{B}_\mu Q$ is a strict contraction. As a result, by the contraction mapping theorem \cite{SELL197342}, a unique GASE exists. Furthermore by \cite[Theorem 3.1]{563625}, $Q_t$ will converge to it. We denote this unique GASE by $Q^*_\mu$. Then,
    \begin{align*}
        \|Q^*_\mu - Q^*_{\mu'}\|_\infty &= \|\mathcal{B}_\mu Q^*_\mu - \mathcal{B}_{\mu'} Q^*_{\mu'}\|_\infty\\
        &\leq \|\mathcal{B}_\mu Q^*_\mu - \mathcal{B}_{\mu'} Q^*_{\mu}\|_\infty + \|\mathcal{B}_{\mu'} Q^*_\mu - \mathcal{B}_{\mu'} Q^*_{\mu'}\|_\infty\\
        &\leq \|\mathcal{T}_2(Q^*_\mu,\mu) - \mathcal{T}_2(Q^*_\mu,\mu')\|_\infty + \gamma \|Q^*_{\mu}-Q^*_{\mu'}\|_\infty\\
        &\leq |\mathcal{X}|(L_f + \gamma L_p\|Q^*_\mu\|_\infty)\|\mu-\mu'\|_\infty + \gamma \|Q^*_{\mu}-Q^*_{\mu'}\|_\infty\\
        &\leq |\mathcal{X}|(L_f + \frac{\gamma}{1-\gamma} L_p\|f\|_\infty)\|\mu-\mu'\|_\infty + \gamma \|Q^*_{\mu}-Q^*_{\mu'}\|_\infty.
    \end{align*}
    As a result, we have $\|Q^*_\mu - Q^*_{\mu'}\|_\infty \leq \frac{L_f|\mathcal{X}| + \frac{\gamma}{1-\gamma} L_p\|f\|_\infty|\mathcal{X}|}{1-\gamma}\|\mu-\mu'\|_\infty$. So $Q^*_\mu$ is uniformly Lipschitz with respect to $\mu$.
\end{proof}
In what follows we make the additional assumptions:
\begin{assumption}\label{mfclp}
    We assume $L_p < |\mathcal{X}|c_{min}$, in particular $c_{min}>0$, where $c_{min}$ is defined in  \eqref{def:cmin}.
\end{assumption}
\begin{assumption}\label{GASE}
The first ODE in system \eqref{MFGODE} along the GASE of the second ODE, that is $\dot{\mu}_t = \mathcal{P}_2(Q^*_{\mu_t},\mu_t)$, has a unique GASE, that we will denote by $\mu^{*\phi}$.
\end{assumption}

In Appendix \ref{GASEproof}, we show how to derive the existence and uniqueness of such a GASE using a sufficient condition on $\phi$ which allows us to employ a contraction argument. However, this argument requires an upper bound on the choice of the parameter $\phi$ which limits its use in our context of deriving the convergence of our algorithm to a solution of the MFG problem. Therefore, in what follows we work under Assumption \ref{GASE} which could alternatively be verified by means of Lyapunov functions for instance.

We can now introduce our first theorem, which guarantees the convergence of the idealized two-timescale approach. It relies on the following assumptions about the update rates:
\begin{assumption}\label{squaresumlearningrate_mfg} The learning rate $\rho^Q_n$ and $\rho^{\mu}_n$ are sequences of positive real numbers satisfying
\begin{align*}
    \sum_n \rho^Q_n = \sum_n \rho^{\mu}_n = \infty, \qquad \sum_n |\rho^Q_n|^2 + |\rho^{\mu}_n|^2 <\infty,
    \qquad \rho^\mu_n/\rho^Q_n \xrightarrow[n\to+\infty]{} 0.
\end{align*}
\end{assumption}
\begin{theorem}\label{thMFGcvg}
    Under Assumptions~\ref{fp_lipschitz}, \ref{mfclp}, \ref{GASE} and \ref{squaresumlearningrate_mfg}, $(\mu_n, Q_n)$ defined in \eqref{simpletwotimescaleapproach_mfg} converges  as $n \to\infty$. We denote the limit by $(\mu^{*\phi},Q^*_{\mu^{*\phi}})$.
\end{theorem}
\begin{proof}
    With Assumptions \ref{fp_lipschitz}, \ref{mfclp}, \ref{GASE} and \ref{squaresumlearningrate_mfg}, and the results of Propositions \ref{PLipschitz}, \ref{TLipschitz}, and \ref{eQ_MFG}, the assumptions of \cite[Theorem 1.1]{Borkar97} are satisfied. This result guarantees the convergence in the statement.
\end{proof}
Next, we show that the limit point is an approximation of the MFG solution. For that matter, we introduce the additional assumption:

\begin{assumption}\label{contraction}
We assume the following system of equations has a unique solution:
    \begin{align}
    \label{eq:mfg-argmin}
        \begin{cases}
            \hat{\mu}=\hat{\mu}\mathrm{P}^{\hat\pi, \hat{\mu}}\\
            Q^*_{\hat{\mu}} = Q^{\hat\pi}_{\hat{\mu}}\\
            \hat\pi = \argmine Q^*_{\hat{\mu}}.   
        \end{cases}
    \end{align}
\end{assumption}

\begin{remark}
    In the literature, the question of non-uniqueness of a best response to a given mean field (and in particular, to the equilibrium mean field) is known to be challenging. A usual approach is to regularized the problem by introducing a penalty in a way that helps to ensure uniqueness of the optimal policy, see e.g.~\cite{wang2020reinforcement} in continuous time RL. We refer for instance to~\cite{hadikhanloo2019finite,cui2021approximately,lauriere2022learning,guo2022entropy,anahtarci2023q} in the literature on learning mean field games. For instance, one can incorporate an entropy penalization in the reward, which encourages the policy to be distributed over all the actions, while keeping larger probabilities on the optimal actions. This is related to optimizing over soft-min policies, as we do in this paper. We leave the detailed analysis of the connection between the two approaches for future work.
\end{remark}

The interpretation of~\eqref{eq:mfg-argmin} is the following. 
The first equation says that $\hat{\mu}$ is the stationary distribution obtained when the whole population is using the policy 
    $\hat{\pi} = \argmine Q^*_{\hat{\mu}}$. The second equation says that $Q^*_{\hat{\mu}}$ is the optimal value function of a representative player. Overall, this implies that $(\hat\mu,\hat{\pi})$ form an MFG equilibrium (see Definition~\ref{def:MFGeq}).

\begin{definition}\label{MFGC} Construction of an MFG solution 
     from $(\mu^{*\phi},Q^*_{\mu^{*\phi}})$ obtained in Theorem \ref{thMFGcvg} as the limit of the algorithm \eqref{simpletwotimescaleapproach_mfg}.
     \begin{enumerate}
         \item Set $\pi^* = \argmine Q^*_{\mu^{*\phi}}$
         \item Calculate $\mu^*$ by using $\mu^*=\mu^*\mathrm{P}^{\pi^*, \mu^*}$
         \item For fix $\mu^*$, compute $Q^*$ by using Bellman equation \eqref{Bellman_MFG}.
\end{enumerate}
\end{definition}
\begin{remark}
    If $\argmine Q^*_{\mu^{*\phi}} = \argmine Q^*$, then we have $(\mu^*,Q^*,\pi^*)$ satisfies \eqref{eq:mfg-argmin}. By the uniqueness assumed in Assumption~\ref{contraction}, we can conclude $(\mu^*,\pi^*)$ is a solution in the sense of definition \ref{def:MFGeq}.
\end{remark}
The following result shows that $(\mu^{*\phi},Q^*_{\mu^{*\phi}})$ are close to $(\mu^*,Q^*)$.

\begin{theorem}\label{th:accuracyMFG}
    Suppose Assumptions~\ref{fp_lipschitz}, \ref{mfclp}, \ref{GASE}, \ref{squaresumlearningrate_mfg} and \ref{contraction} hold.
Let $(\hat{\mu}, Q^*_{\hat\mu})$ be the solution of~\eqref{eq:mfg-argmin}.
    Let $\delta(\phi)$ be the action gap defined as $\delta(\phi) = \min_{x\in\mathcal{X}}(\min_{a\notin \argmin_a Q^*_{\mu^{*\phi}}}Q^*_{\mu^{*\phi}}(x,a) -\min_a Q^*_{\mu^{*\phi}}(x,a)) > 0$, and $\delta(\phi) =\infty$ if $Q^*_{\mu^{*\phi}}(x)$ is constant with respect to $a$ for each $x$. 
    Then:  
    \begin{align}\label{mfgerror}
    \begin{cases}
        \|\mu^{*\phi} - \mu^*\|_1
        \leq 
        \frac{2|\mathcal{A}|^{\frac{3}{2}}\exp(-\phi\delta(\phi))}{|\mathcal{X}|c_{min}-L_p}
        \\
        \|Q^*_{\mu^{*\phi}} - Q^*\|_\infty
        \leq 
        \frac{(L_f + \frac{\gamma}{1-\gamma} L_p\|f\|_\infty)2|\mathcal{A}|^{\frac{3}{2}}\exp(-\phi\delta(\phi))}{(1-\gamma)(|\mathcal{X}|c_{min}-L_p)}.
    \end{cases}
    \end{align}
\end{theorem}

\begin{proof}
    Using~\eqref{diff-P-policies} for the first term of the second inequality below, we have:
\begin{align*}
     \|\mu^{*\phi} - \mu^*\|_1 &= \|\mu^{*\phi}\mathrm{P}^{\softmin_\phi Q^{*\phi}_{\mu^{*\phi}},\mu^{*\phi}} - \mu^*\mathrm{P}^{\argmine Q^{*\phi}_{\mu^{*\phi}},\mu^*}\|_1 \\
&\leq\|\mu^{*\phi}\mathrm{P}^{\softmin_\phi Q^{*\phi}_{\mu^{*\phi}},\mu^{*\phi}} - \mu^{*\phi}\mathrm{P}^{\argmine Q^{*\phi}_{\mu^{*\phi}},\mu^{*\phi}}\|_1+\|\mu^{*\phi}\mathrm{P}^{\argmine Q^{*\phi}_{\mu^{*\phi}},\mu^{*\phi}} - \mu^*\mathrm{P}^{\argmine Q^{*\phi}_{\mu^{*\phi}},\mu^*}\|_1 \\
&\leq|\mathcal{A}|^{\frac{1}{2}}\max_{x\in\mathcal{X}}\|{\softmin}_\phi Q^*_{\mu^{*\phi}}(x)-\argmine Q^*_{\mu^{*\phi}}(x)\|_2 + (L_p+1-|\mathcal{X}|c_{min})\|\mu^{*\phi} - \mu^*\|_1  \\
    &\leq
    2|\mathcal{A}|^{\frac{3}{2}}\exp(-\phi\delta(\phi))+(L_p+1-|\mathcal{X}|c_{min})\|\mu^{*\phi} -\mu^*\|_1 ,
\end{align*}
where the first term in the last inequality comes from~\cite[Lemma 7]{guo2019learning}, which bounds the distance between $\softmin_\phi$ and $\argmine$. 
Consequently, we obtain the first inequality in~\eqref{mfgerror}.

Similarly, 
\begin{align*}
    \|Q^*_{\mu^{*\phi}} - Q^*\|_\infty &= \|\mathcal{B}_{\mu^{*\phi}} Q^*_{\mu^{*\phi}} - \mathcal{B}_{\mu^*} Q^*\|_\infty \\
    &\leq \|\mathcal{B}_{\mu^{*\phi}} Q^*_{\mu^{*\phi}} - \mathcal{B}_{\mu^{*\phi}} Q^*\|_\infty + \|\mathcal{B}_{\mu^{*\phi}} Q^* - \mathcal{B}_{\mu^*} Q^*\|_\infty \\
    &\leq \gamma \|Q^*_{\mu^{*\phi}} - Q^*\|_\infty + (L_f + \frac{\gamma}{1-\gamma} L_p\|f\|_\infty)\|\mu^{*\phi}-\mu^*\|_1,
\end{align*}
so that we obtain the second inequality in~\eqref{mfgerror}. 
\end{proof}
We are now ready for the main result of this section.
\begin{theorem} \label{th:mainMFG}
Under
Assumptions~\ref{fp_lipschitz}, \ref{mfclp}, \ref{GASE}, \ref{squaresumlearningrate_mfg}, \ref{contraction}, and additionally $\delta = \liminf_{\phi\to\infty}\delta(\phi)>0$, 
the distribution and policy $(\mu^*,\pi^*)$ introduced in Definition \ref{MFGC}  form a solution to the MFG problem (Definition \ref{def:MFGeq}).
\end{theorem}
\begin{proof}
From algorithm \eqref{simpletwotimescaleapproach_mfg} we have, for all $(x,a)$, 
\[
    Q^*_{\mu^{*\phi}}(x,a)=f(x,a,\mu^{*\phi})+\gamma\sum_{x'}p(x'|x,a,\mu^{*\phi}) \min_{a'} Q^*_{\mu^{*\phi}}(x,a'), 
\]
and by definition of $\pi^*$, 
\[
    Q^*_{\mu^{*\phi}}(x,a) > \min_{a'} Q^*_{\mu^{*\phi}}(x,a')\quad \text{for}\quad a\notin \mbox{support}(\pi^*(x)).
\]
The difference between the left-hand side and the right-hand side above is greater than the gap $\delta(\phi)$. Under the condition that the errors in 
\eqref{mfgerror} are small enough (by choosing  $\phi$ large enough), we can replace $Q^*_{\mu^{*\phi}}$ by $Q^*$ and $\mu^{*\phi}$ by $\mu^*$ and obtain:
\[
Q^*(x,a) = f(x,a,\mu^{*})+\gamma\sum_{x'}p(x'|x,a,\mu^{*})\min_{a'} Q^*(x,a') > \min_{a'} Q^*(x,a')\quad \text{for}\quad a\notin \mbox{support}(\pi^*(x)),
\] 
i.e. $\argmine Q^*= \argmine Q^*_{\mu^{*\phi}}$, so that $(\mu^*,\pi^*)$ is indeed an MFG solution.
\end{proof}

\section{Mean Field Control}\label{sec:MFC}

We still use the notations introduced in Section~\ref{sec:notations} but we introduce a different notion of Q-function.

\subsection{Mean Field Control Formulation}\label{sec:MFCformulation}
In the context of Asymptotic MFC introduced in \cite[Section 2.2]{Andrea20}, we can consider modified Q-functions and population distributions. For an admissible policy $\pi$, we define the McKean--Vlasov-dynamics, MKV-dynamics for short, by $p(x'|x,a,\mu^\pi)$ so that $\mu^\pi$ is the limiting distribution of the associated process $(X_n^\pi)$. We define the policy $\tilde{\pi}^{(x,a)}$ by
\begin{align}\label{alphacontrol}
    \tilde{\pi}^{(x,a)} (x') := \begin{cases}
        \delta_a\hspace{3mm} \text{if}\hspace{5mm} x'=x, \\
        \pi(x)\hspace{2mm} \text{for}\hspace{3mm} x'\neq x.
    \end{cases} 
\end{align}
Note that the policy $\tilde{\pi}^{(x,a)}$ depends on $(x,a)$ although most of the time we will omit  this explicit dependence  in the notation. 
Then, the modified Q-function is given by
\begin{align}\label{MFCQ}
    Q^\pi(x,a) := f(x,a,\mu^{\tilde{\pi}}) + \mathbbm{E}\left[\sum_{n\geq 1}\gamma^n f(X_n^{\pi}, A_n,\mu^\pi)\right],
\end{align}
where, for a given $(x,a)$, $\mu^{\tilde{\pi}}$ is the limiting distribution of $X^{\tilde{\pi}}_n$, and where
the expectation is along the Markovian dynamics
\begin{align}\label{mfc_dynamic}
    \begin{cases}
    X_0^{\pi} = x, A_0 = a,\\ 
    X_{n+1}^{\pi} \sim p(\cdot|X_n^{\pi}, A_n, \mu^\pi),\\
    A_n \sim {\pi}(X_n^{\pi}).
    \end{cases}
\end{align}
The optimal Q-function  is defined as:
\begin{align}
    Q^*(x,a) := \inf_{\pi} Q^{\pi}(x,a).
\end{align} 

\begin{assumption}\label{uniqueness minimum}
    We assume that there exists a unique policy $\pi^*$ such that: $Q^*(x,a) := Q^{\pi^*}(x,a)$ for every $(x,a)$. 
    We also assume that for every $x \in \mathcal{X}$, $\{a \in \mathcal{A} : Q^*(x,a) = \min_{a'} Q^*(x,a')\}$ is a singleton, whose element is a pure control denoted by $\alpha^*(x) = \argmin_{a'} Q^*(x,a') \in \mathcal{A}$. In other words, the policy $\pi^*$ is a pure control $\alpha^*$.
\end{assumption}

Under Assumption \ref{uniqueness minimum}, by Theorem 2 in \cite{Andrea20}, $Q^*$ satisfies the McKean--Vlasov Bellman equation:
\begin{align}\label{Bellman_MFC}
    Q^* (x,a) = f(x,a,\tilde{\mu}^*)+\gamma \sum_{x'\in\mathcal{X}}p(x'|x,a,\tilde{\mu}^*)\min_{a'}Q^*(x',a'), 
\end{align}
where 
$\tilde{\mu}^*:=\mu^{\tilde{\pi}^*}$ is the invariant distribution associated to the policy $\tilde\pi^*$ which implicitly depends on $(x,a)$ and is defined by (\ref{alphacontrol}).

In short, we denote the Bellman equation (\ref{Bellman_MFC}) by $Q^* = \mathcal{B} Q^*$, where $\mathcal{B}$ is defined for every $Q^{\pi}$  and the limiting distribution $\mu^{\tilde\pi}$ induced by the policy $\tilde{\pi}$ by:
\begin{equation}\label{eq:BQ}
    \mathcal{B} Q^{\pi}(x,a) := f(x,a,\mu^{\tilde\pi})+\gamma \sum_{x'\in\mathcal{X}}p(x'|x,a,\mu^{\tilde\pi})\min_{a'}Q^{\pi}(x',a').
\end{equation}

As in the case of MFGs, we introduce the operator $\mathcal{B}_\mu$ defined for every matrix $Q$ and every probability distribution $\mu$, by
\begin{align}\label{def:Bmu}
    \mathcal{B}_\mu Q (x,a) := f(x,a,\mu)+\gamma \sum_{x'\in\mathcal{X}}p(x'|x,a,\mu)\min_{a'}Q(x',a').
\end{align}
So $\mathcal{B}$ given by \eqref{eq:BQ}, and $\mathcal{B}_{\mu}$ are related by:
$\mathcal{B}Q^\pi=\mathcal{B}_{\mu^{\tilde{\pi}}}Q^\pi$.

\begin{definition}\label{MFCV}
    Let $V^\pi : \mathcal{X}\mapsto\mathbbm{R}$ be the value function for the policy $\pi$ defined as: 
    $$
    V^\pi(x) = \mathbbm{E}\left[\sum_{n\geq 0}\gamma^n f(X_n^{\pi}, A_n,\mu^\pi)|X_0^{\pi}=x\right]
    $$
following the dynamics \eqref{mfc_dynamic}.
\end{definition}

If we consider a policy $\pi$ which corresponds to a pure control $\alpha$, we have:
\begin{lemma} \cite[Lemma 3]{Andrea20} Letting $\pi(x) = \delta_{\alpha(x)}$, we have:
\begin{align*} 
    V^{\pi}(x) &= Q^{\pi}(x,\alpha(x)).
  \end{align*}
\end{lemma}
In turn, we define the optimal value function and we recall:
\begin{theorem} \cite[Theorem 5]{Andrea20}
\begin{align*} 
 V^*(x) &:=\inf_{\pi} V^{\pi}(x)= \min_a Q^*(x,a).
 \end{align*}
 \end{theorem}
 Note that by taking $\delta_{\alpha^*(x)}=\pi^*(x)$ in \eqref{Bellman_MFC}, one obtains the Bellman equation for  $V^*=V^{\pi^*}$:
 \begin{align} \label{eq:VMFC}
 V^*(x)=f(x,\pi^*(x),\mu^*)+\gamma \sum_{x'\in\mathcal{X}} p(x'|x,\pi^*(x),\mu^*)V^*(x'),
 \end{align}  
noting that $\pi^*$ and $\mu^*$ depend on $V^*$.

\begin{definition}\label{def:MFCsolution}
A pure control policy $\alpha^*$ is a solution to the MFC problem if 
\[
V^*(x):=V^{\alpha^*}(x)=\inf_\pi V^\pi(x).
\]
\end{definition}
\begin{lemma}\label{lemmaCharacterization}
    The solution $\alpha^*$  of the MFC problem can be characterized as follows:
    \begin{enumerate}
        \item 
        Define $Q^*$ by
        \[
        Q^*(x,a)=Q^{\alpha^*}(x,a)=f(x,a,\mu^{\tilde{\alpha}^*})+\gamma\sum_{x'}p(x'|x,a,\mu^{\tilde\alpha^*})V^*(x').
        \]
        \item Check that $\alpha^*(x)=\argmin_a Q^*(x,a)$, so that $Q^*$ satisfies the Bellman equation \eqref{Bellman_MFC}.
    \end{enumerate}
\end{lemma}

\subsection{Idealized Two-timescale Iteration Approach for MFC}

We will solve the problem defined above using again a two-timescale approach as for MFG in Section \ref{sec:MFGsimple}, but this time with learning rates in the regime $\rho^Q_n/\rho^{\mu}_n\to 0$ as $n\to+\infty$. In the MFC case the idealized algorithm should keep track of a vector of probability distributions indexed by $(x,a)$, i.e., there is one distribution $\mu^{(x,a)}$ for each $(x,a)$.

The finite difference equation system for the MFC problem becomes:
\begin{align}\label{simpletwotimescaleapproach_mfc}
    \begin{cases}
        \mu_{n+1}^{(x,a)} = \mu_n^{(x,a)} + \rho_n^{\mu}\tilde{\mathcal{P}}_2^{(x,a)}(Q_n,\mu_n^{(x,a)}),\\
        Q_{n+1}(x,a) = Q_n(x,a) + \rho^Q_n\mathcal{T}_2(Q_n,\mu_n^{(x,a)})(x,a),
    \end{cases}
\end{align}
where the functions $\tilde{\mathcal{P}}^{(x,a)}_2: \mathbb{R}^{|\mathcal{X}|\times|\mathcal{A}|} \to \Delta^{|\mathcal{X}|}$ and $\mathcal{T}_2: \mathbb{R}^{|\mathcal{X}|\times|\mathcal{A}|} \to \mathbb{R}^{|\mathcal{X}| \times |\mathcal{A}|}$ are defined by:
\begin{align*}
    \begin{cases}
\tilde{\mathcal{P}}^{(x,a)}_2(Q,\mu)(y) = \sum_{x'\neq x}\mu(x')p(y|x',{\softmin}_\phi Q(x'),\mu) +\mu(x)p(y|x,a,\mu) - \mu(y),\\
        \mathcal{T}_2(Q,\mu)(x,a) = f(x,a,\mu) + \gamma\sum_{x'}p(x'|x,a,\mu)\min_{a'}Q(x',a')-Q(x,a).
    \end{cases}
\end{align*}
and where the first function can also be written as 
\begin{align*}
    \mathcal{\tilde P}^{(x,a)}_2 (Q,\mu) = \mu\mathrm{\tilde P}_{(x,a)}^{\softmin_\phi Q, \mu} - \mu,
\end{align*}
which defines $\mathrm{\tilde P}_{(x,a)}^{\softmin_\phi Q, \mu}$.

Roughly speaking, by~\cite{Borkar97}, in this regime, the solution $(\mu_n,Q_n)_n$ to the 
above iterations~\eqref{simpletwotimescaleapproach_mfc} has a behavior that is described (in a sense to be made precise later) by the following system of ODEs, in the regime where $\epsilon$ is small and $t$ goes to infinity: 
\begin{align}\label{MFCODE}
\begin{cases}
    {\dot{\mu}^{(x,a)}_t} = \frac{1}{\epsilon}\tilde{\mathcal{P}}^{(x,a)}_2(Q_t,\mu^{(x,a)}_t),\\
    \dot{Q}_t(x,a) = \mathcal{T}_2(Q_t,{\mu}^{(x,a)}_t)(x,a),
\end{cases}
\end{align}
where $\rho^Q_n/\rho^{\mu}_n$ is thought of being of order $\epsilon \ll 1$.

\subsection{Convergence of the Idealized Two-timescale Iteration for MFC}
In this section, we  establish the convergence of the idealized two-timescale approach defined in \eqref{simpletwotimescaleapproach_mfc}. In order to do it, let us introduce our additional assumptions: 
\begin{assumption}\label{squaresumlearningrate_mfc} The learning rate $\rho^Q_n$ and $\rho^{\mu}_n$ are sequences of positive real numbers satisfying
\begin{align*}
    \sum_n \rho^Q_n = \sum_n \rho^{\mu}_n = \infty, \qquad \sum_n |\rho^Q_n|^2 + |\rho^{\mu}_n|^2 <\infty,
    \qquad \rho^Q_n /\rho^{\mu}_n\xrightarrow[n\to+\infty]{} 0.
\end{align*}
\end{assumption}

With Assumption~\ref{fp_lipschitz}, we preserve Proposition~\ref{PLipschitz} and Proposition~\ref{TLipschitz} which become:

\begin{proposition}\label{TPLipschitz}
If Assumption~\ref{fp_lipschitz} holds, then
    the function $(Q,\mu) \mapsto \mathcal{\tilde P}^{(x,a)}_2(Q,\mu)$ is Lipschitz with respect to both $Q$ and $\mu$: 
\begin{align}\label{LipTP2Q}
        \|\mathcal{\tilde P}^{(x,a)}_2(Q,\mu) - \mathcal{\tilde P}^{(x,a)}_2(Q',\mu)\|_{\infty}
\leq \phi|\mathcal{A}|\|Q-Q'\|_\infty,
    \end{align}
and
     \begin{align}\label{LipTP2mu}
        \|\mathcal{\tilde P}^{(x,a)}_2(Q,\mu) - \mathcal{\tilde P}^{(x,a)}_2(Q,\mu')\|_{\infty}
         \leq (L_p + 2-|\mathcal{X}|c_{min})\|\mu-\mu'\|_1.
    \end{align}
    \end{proposition}
\begin{proof}
We have:
\begin{align*}
        \|\mathcal{\tilde P}^{(x,a)}_2(Q,\mu) - \mathcal{\tilde P}^{(x,a)}_2(Q',\mu)\|_{\infty} 
        &= \|\mu\mathrm{\tilde P}_{(x,a)}^{\softmin_\phi Q, \mu} - \mu\mathrm{\tilde P}_{(x,a)}^{\softmin_\phi Q', \mu}\|_\infty\\
    &\leq\|\mu\mathrm{P}^{\softmin_\phi Q, \mu} - \mu\mathrm{P}^{\softmin_\phi Q', \mu}\|_{1}\\
        &\leq \phi|\mathcal{A}|\|Q-Q'\|_\infty,
\end{align*}    
    and,
\begin{align*}
        \|\mathcal{\tilde P}^{(x,a)}_2(Q,\mu) - \mathcal{\tilde P}^{(x,a)}_2(Q,\mu')\|_{\infty} &\leq\|\mu\mathrm{\tilde P}_{(x,a)}^{\softmin_\phi Q, \mu} - \mu'\mathrm{\tilde P}_{(x,a)}^{\softmin_\phi Q, \mu'}\|_\infty + \|\mu-\mu'\|_\infty\\
        &\leq \|\mu\mathrm{P}^{\softmin_\phi Q, \mu}-\mu'\mathrm{P}^{\softmin_\phi Q, \mu'}\|_{1} + \|\mu-\mu'\|_{1} \\
        &\leq (L_p + 2-|\mathcal{X}|c_{min})\|\mu-\mu'\|_1.
\end{align*}
\end{proof}

 Now that we have this Lipschitz property, we show that the ODEs for $\mu^{(x,a)}_t$ in the system \eqref{MFCODE} have  unique global asymptotically stable equilibria (GASE).

\begin{proposition}\label{emu_MFC}
    Suppose Assumption \ref{mfclp} hold. For any given $Q$, ${\dot{\mu}_t} = \tilde{\mathcal{P}}_2(Q,\mu_t)$ has a unique GASE, that will denote by ${\mu}^{*\phi}_Q$ (where we omit the dependency on $(x,a)$). Moreover, ${\mu}_Q^{*\phi}:\mathbbm{R}^{|\mathcal{X}|\times|\mathcal{A}|}\to\Delta^{|\mathcal{X}|}, Q \mapsto \mu^{*\phi}_Q$ is Lipschitz.
\end{proposition}
\begin{proof}
    By Proposition \ref{TPLipschitz}, we have $ \|\mu\mathrm{\tilde P}^{\softmin_\phi Q, \mu}-\mu'\mathrm{\tilde P}^{\softmin_\phi Q, \mu'}\|_1 \leq (L_p+1-|\mathcal{X}|c_{min})\|\mu-\mu'\|_1$, where $L_p + 1-|\mathcal{X}|c_{min} < 1$ implies that $\mu \mapsto \mathrm{\tilde P}^{\softmin_\phi Q, \mu}\mu$ is a strict contraction. As a result, by contraction mapping theorem \cite{SELL197342}, a unique GASE exists and furthermore, by \cite[Theorem 3.1]{563625}, it is the limit of $\mu_t$ denoted by ${\mu}^{*\phi}_Q$ (at each $(x,a)$). Then,
    \begin{align*}
        \|\mu^{*\phi}_Q - \mu^{*\phi}_{Q'}\|_1 &\leq \|\mu^{*\phi}_Q\mathrm{P}^{\softmin_\phi Q, \mu^{*\phi}_Q} - \mu^{*\phi}_{Q'}\mathrm{P}^{\softmin_\phi Q', \mu^{*\phi}_{Q'}}\|_1\\
        &\leq \|\mu^{*\phi}_Q\mathrm{P}^{\softmin_\phi Q, \mu^{*\phi}_Q} - \mu^{*\phi}_{Q'}\mathrm{P}^{\softmin_\phi Q, \mu^{*\phi}_{Q'}}\|_1 + \|\mu^{*\phi}_{Q'}\mathrm{P}^{\softmin_\phi Q, \mu^{*\phi}_{Q'}} - \mu^{*\phi}_{Q'}\mathrm{P}^{\softmin_\phi Q', \mu^{*\phi}_{Q'}}\|_1\\
        &\leq \|\mathcal{P}_2(Q,\mu^{*\phi}_{Q'}) - \mathcal{P}_2(Q',\mu^{*\phi}_{Q'})\|_1 + (L_p+1-|\mathcal{X}|c_{min})\|\mu^{*\phi}_Q - \mu^{*\phi}_{Q'}\|_1\\
        &\leq \phi|\mathcal{A}|\|Q-Q'\|_1 + (L_p+1-|\mathcal{X}|c_{min})\|\mu^{*\phi}_Q - \mu^{*\phi}_{Q'}\|_1.
    \end{align*}
    As a result, we have $\|\mu^{*\phi}_Q - \mu^{*\phi}_{Q'}\|_1 \leq \frac{\phi|\mathcal{A}|}{|\mathcal{X}|c_{min}-L_p}\|Q - Q'\|_1$, which is the uniform Lipschitz property of $\mu^{*\phi}_Q$ with respect to $Q$ for each $(x,a)$. 
\end{proof}

\begin{assumption}\label{GASE-MFC}
The second ODE in the system \eqref{MFCODE} along the GASE of the first ODEs, that is
$\dot{Q}_t = \mathcal{T}_2(Q_t,{\mu^{*\phi}_{Q_t}})$, has a unique GASE that we will denote by $Q^{*\phi}$.
\end{assumption}

In Appendix \ref{GASEproofMFC}, we show how to derive the existence and uniqueness of such a GASE using a contraction argument. However, as in the MFG case, this argument requires an upper bound on the choice of the parameter $\phi$ which limits its use in our context of deriving the convergence of our algorithm to a solution of the MFC problem. Therefore, in what follows we work under Assumption \ref{GASE-MFC}  which could alternatively be verified by means of Lyapunov functions for instance.

Let $V^{*\phi}(x) = \min_{a'} Q^{*\phi}(x,a')$ for all $x\in \mathcal{X}$. From \eqref{simpletwotimescaleapproach_mfc}, we define $V_n(x) = \min_{a'} Q_n(x,a')$ for all $x\in \mathcal{X}$, then we have the following convergence result.

\begin{theorem}\label{thMFCcvg}
    Suppose Assumptions \ref{fp_lipschitz}, \ref{mfclp}, \ref{squaresumlearningrate_mfc} and \ref{GASE-MFC} hold. Then, $(\mu_n, V_n)$ defined above converges to $(\mu^{*\phi}_{Q^{*\phi}},V^{*\phi})$ as $n \to\infty$, where we recall that $\mu_n$ and $\mu^{*\phi}_{Q^{*\phi}}$ depend on $(x,a)$.
\end{theorem}
\begin{proof}
    With Assumption \ref{fp_lipschitz}, \ref{mfclp},  \ref{squaresumlearningrate_mfc}, \ref{GASE-MFC},  and the results of Propositions \ref{PLipschitz}, \ref{TLipschitz}, and \ref{emu_MFC}, the assumptions of \cite[Theorem 1.1]{Borkar97} are satisfied, and therefore  guarantees the convergence in the statement.
\end{proof}

Next, we show that the limit point given by the algorithm, that is $(\mu^{*\phi}_{Q^{*\phi}},V^{*\phi})$, provides an approximation of an MFC solution as follows. 

\begin{definition}\label{MFCC} Construction of an MFC solution:
    \begin{enumerate}
         \item Set $\alpha^*(x) = \argmin_{a'} Q^{*\phi}(x,a')$
         \item 
         For each $(x,a)$, define $\tilde{\alpha}^*$ by \eqref{alphacontrol} and compute $\mu^{\tilde{\alpha}^*}$ solution to
         \[
         \mu(y)=
             \sum_{x'\neq x} \mu (x')  p(y|x',\alpha^*(x'),\mu)+\mu(x)p(y| x,a, \mu),\quad \forall y.
         \]
        \item Compute $V^{\alpha^*}$ by using Definition \ref{MFCV}.
\end{enumerate}
\end{definition}

The following result shows that $(\mu^{*\phi}_{Q^{*\phi}},V^{*\phi})$ are close to $(\mu^{\tilde{\alpha}^*},V^{\alpha^*})$.

\begin{theorem} \label{mfc_main}
Suppose Assumptions \ref{fp_lipschitz}, \ref{mfclp},  \ref{squaresumlearningrate_mfc} and \ref{GASE-MFC} hold.
Let $(\mu^*, V^{\alpha^*})$ given by Definition \ref{MFCC}.
    Let $\delta(\phi)$ be the action gap defined as $\delta(\phi) = \min_{x\in\mathcal{X}}(\min_{a\notin \argmin_a Q^{*\phi}}Q^{*\phi}(x,a) -\min_a Q^{*\phi}(x,a)) > 0$, and $\delta(\phi) =\infty$ if $Q^{*\phi}(x)$ is constant with respect to $a$ for each $x$. 
    Then, 
    \begin{align}\label{mfc_errorphi}
    \begin{cases}
        \|\mu^{*\phi}_{Q^{*\phi}} - \mu^*\|_1
        \leq 
        \frac{2|\mathcal{A}|^{\frac{3}{2}}\exp(-\phi\delta(\phi))}{|\mathcal{X}|c_{min}-L_p},
        \\
        \|V^{*\phi} - V^{\alpha^*}\|_\infty
        \leq 
        \frac{(L_f + \frac{\gamma}{1-\gamma} L_p\|f\|_\infty)2|\mathcal{A}|^{\frac{3}{2}}\exp(-\phi\delta(\phi))}{(1-\gamma)(|\mathcal{X}|c_{min}-L_p)}.
    \end{cases}
    \end{align}
\end{theorem}

\begin{proof}
We follow the same lines as in the proof of Theorem~\ref{th:accuracyMFG}. We have:
\begin{align*}
     \|\mu^{*\phi}_{Q^{*\phi}} - \mu^{\tilde{\alpha}^*}\|_1 %
&\leq|\mathcal{A}|^{\frac{1}{2}}\max_{x\in\mathcal{X}}\|{\softmin}_\phi Q^{*\phi}(x)-\argmin Q^{*\phi}(x)\|_2 + (L_p+1-|\mathcal{X}|c_{min})\|\mu^{*\phi}_{Q^{*\phi}} - \mu^*\|_1  \\
    &\leq
    2|\mathcal{A}|^{\frac{3}{2}}\exp(-\phi\delta(\phi))+
    (L_p+1-|\mathcal{X}|c_{min})\|\mu^{*\phi} -\mu^{\tilde{\alpha}^*}\|_1 ,
\end{align*}
which gives the first inequality in~\eqref{mfc_errorphi}. %

Similarly, 
\begin{align*}
    \|V^{*\phi} - V^{\alpha^*}\|_\infty &= \|\mathcal{B}_{\mu^{*\phi}_{Q^{*\phi}}} V^{*\phi} - \mathcal{B}_{\mu^*} V^{\alpha^*}\|_\infty \\
    &\leq \|\mathcal{B}_{\mu^{*\phi}_{Q^{*\phi}}} V^{*\phi} - \mathcal{B}_{\mu^{*\phi}_{Q^{*\phi}}} V^{\alpha^*}\|_\infty + \|\mathcal{B}_{\mu^{*\phi}_{Q^{*\phi}}} V^{\alpha^*} - \mathcal{B}_{\mu^*} V^{\alpha^*}\|_\infty \\
    &\leq \gamma \|V^{*\phi} - V^{\alpha^*}\|_\infty + (L_f + \frac{\gamma}{1-\gamma} L_p\|f\|_\infty)\|\mu^{*\phi}_{Q^{*\phi}}-\mu^*\|_1,
\end{align*}
which gives the second inequality in~\eqref{mfc_errorphi}. 
\end{proof}
We are now ready for the main result of this section. 
\begin{theorem}\label{th:mainMFC}
Under the assumptions in Theorem \ref{mfc_main}, and additionally  $\delta=\liminf_{\phi\to\infty}\delta(\phi)
>0$,
the pure policy $\alpha^*$ defined by item 1 in Definition \ref{MFCC}  is a solution to the MFC problem characterized in Lemma \ref{lemmaCharacterization}.
\end{theorem}
\begin{proof}
We use the characterization of a solution given by Lemma \ref{lemmaCharacterization}. Define $Q^*$ by
\[
        Q^*(x,a)=Q^{\alpha^*}(x,a)=f(x,a,\mu^{\tilde{\alpha}^*})+\gamma\sum_{x'}p(x'|x,a,\mu^{\tilde\alpha^*})V^{\alpha^*}(x').
        \]
      Then, we need to check that $\alpha^*(x)=\argmin_a Q^*(x,a)$.

From our algorithm \eqref{simpletwotimescaleapproach_mfc} we have 
\[
f(x,a,\mu^{*\phi}_{Q^{*\phi}})+\gamma\sum_{x'}p(x'|x,a,\mu^{*\phi}_{Q^{*\phi}})V^{*\phi}(x') > V^{*\phi}(x)\quad \text{for}\quad a\neq \alpha^*(x).
\]
The difference between the left-hand side and the right-hand side above is larger than the gap $\delta(\phi)$. Under the condition that the errors in 
\eqref{mfc_errorphi} are small enough (by choosing $\phi$ large enough), we can replace $V^{*\phi}$ by $V^{\alpha^*}$ and $\mu^{*\phi}_{Q^{*\phi}}$ by $\mu^{\tilde{\alpha}^*}$ and obtain:
\[
f(x,a,\mu^{\tilde{\alpha}^*})+\gamma\sum_{x'}p(x'|x,a,\mu^{\tilde\alpha^*})V^{\alpha^*}(x') > V^{\alpha^*}(x)\quad \text{for}\quad a\neq \alpha^*(x),
\] 
i.e. $Q^*(x,a)>V^*(x)$ for $a\neq \alpha^*(x)$. Recalling Lemma \ref{lemmaCharacterization}, $\alpha^*$ is indeed an optimal policy.
\end{proof}

\section{Example}\label{sec:example}
We present a very simple example to illustrate the performance of our algorithm with comparison with the explicit solutions for the MFG and MFC problems.

\subsection{The Model}
Let us consider the state space $\mathcal{X}=\{x_0=0,x_1=1\}$ and the action space $\mathcal{A}=\{stay=0,move=1\}$. The markovian dynamics is given by
\[ 
    p(x'|x,a) = 
    \begin{cases}
        1-p, & \hbox{ if } x'=a(x) 
        \\
        p, & \hbox{ if } x'\ne a(x),
    \end{cases}
\] 
where $0<p<1$
represents a small noise parameter, and where as such we impose $0<p<0.5$. Note that the dynamics does not depend on the population distribution $\mu$ and therefore, in the notation of the previous sections, $L_p=0$.

Let us consider the cost function:
\[
   f(x,a,\mu) = x +  c\mu_0,
\]
where $c>0$ and $\mu_0$ is the probability to be at $x_0$ under the distribution $\mu$ (one may think of $\mu_0$ as $1-\bar{\mu}$). Note that 
\[
    |f(x,a,\mu) - f(x,a,\mu')|
    = c|\mu_0 - \mu'_0| 
    \le L_f \|\mu - \mu'\|_1,
\]
with $L_f=c/2$.

Recall that $c_{min}$ is defined in Proposition~\ref{PLipschitz}. In this example we have:
\[
    c_{min}
    = \min_{x,x',a,\mu} p(x'|x,a,\mu)
    = p>0.
\]      
Therefore, in this example Assumptions~\ref{fp_lipschitz} and \ref{mfclp} are satisfied.

In sections \ref{verifMFG} and \ref{verifMFC} we will show that the other assumptions are satisfied.
Thus, in this example, depending on the learning rates satisfying either Assumption~\ref{squaresumlearningrate_mfg} for MFG or Assumption~\ref{squaresumlearningrate_mfc} for MFC, we can apply 
 either Theorem~\ref{th:mainMFG}  or Theorem~\ref{th:mainMFC}.

\subsection{Theoretical Solutions}
There are 4 pure policies $\alpha$ given by: 
$
\{(s,s),(s,m), (m,s), (m,m)\},
$
associated to the corresponding limiting distributions $\mu^\alpha$: 
$
\{(\frac{1}{2},\frac{1}{2}), (1-p,p), (p,1-p), (\frac{1}{2},\frac{1}{2})\}.
$

\subsubsection {MFG}
For a given distribution $\mu$, the Q-Bellman equation reads:
\[
Q^*_\mu(x,a)=f(x,a,\mu)+\gamma\left[ (1-p)Q^*_\mu(a(x),s)\wedge Q^*_\mu(a(x),m) +p Q^*_\mu(1-a(x),s)\wedge Q^*_\mu(1-a(x),m)\right].
\]
One can check that the optimal strategy is $\alpha^*=(s,m)$ which corresponds to the case 
\[
Q^*_\mu(0,s)<Q^*_\mu(0,m),\quad Q^*_\mu(1,m)<Q^*_\mu(1,s).
\]
This is independent of the distribution $\mu$ which affects $Q$ only as an additive shift, so that the fixed point in the MFG problem is achieved by taking $\mu^*=\mu^{\alpha^*}=\mu^{(s,m)}=(1-p,p)$. The optimal Q-values are given by
\[
Q^*(0,s)=\frac{c(1-p)+\gamma p}{1-\gamma}, \quad Q^*(1,m)=Q^*(0,s) +1,
\]
\[
Q^*(0,m)=\frac{c(1-p)+\gamma^2(2p-1)+\gamma(1-p)}{1-\gamma},\quad Q^*(1,s)=Q^*(0,m)+1,
\]
so that $Q^*(0,m)-Q^*(0,s)=Q^*(1,s)-Q^*(1,m)=\gamma (1-2p)>0$ as $0<p<0.5$. 

The computation above shows that there is unique solution to the MFG, so that Assumption \ref{contraction} is satisfied.

\subsubsection{Verification of Assumption \ref{GASE} and lower bound for the gap in Theorem \ref{th:mainMFG}}\label{verifMFG} From above, for any given $\mu$, we have the following statement
\[
Q^*_\mu(0,s)=\frac{c(1-\mu_1)+\gamma p}{1-\gamma}, \quad Q^*_\mu(1,m)=Q^*_\mu(0,s) +1,
\]
\[
Q^*_\mu(0,m)=\frac{c(1-\mu_1)+\gamma^2(2p-1)+\gamma(1-p)}{1-\gamma},\quad Q^*_\mu(1,s)=Q^*_\mu(0,m)+1,
\]
so that $Q^*_\mu(0,m)-Q^*_\mu(0,s)=Q^*_\mu(1,s)-Q^*_\mu(1,m)=\gamma (1-2p)>0$ as $0<p\ll 1$.\\
Then we can compute the soft-min policy $\pi$ defined by $\pi(a|x)={\softmin}_\phi Q^*_\mu(x)(a)$.
\\
It is independent of $\mu$, so that:
\[
\pi(s|0) = \frac{1}{1+\exp(-\phi\gamma (1-2p))},\quad \pi(m|0)= \frac{\exp(-\phi\gamma (1-2p))}{1+\exp(-\phi\gamma (1-2p))},
\]
\[
\pi(s|1) = \frac{\exp(-\phi\gamma (1-2p))}{1+\exp(-\phi\gamma (1-2p))},\quad \pi(m|1)= \frac{1}{1+\exp(-\phi\gamma (1-2p))}.
\]
Then, we  compute the GASE $\mu^{*\phi}$ in Assumption \ref{GASE} by solving $\mathcal{P}_2(Q^*_{\mu^{*\phi}},\mu^{*\phi}) = 0$. Since we have two states, it is enough to write this equation at one point, say $x=1$.
\begin{align*}
    \mu^{*\phi}_1 & = (1-\mu^{*\phi}_1)\left[\pi(s|0)p(1|0,s) + \pi(m|0)p(1|0,m)\right] + \mu^{*\phi}_1\left[\pi(s|1)p(1|1,s) + \pi(m|1)p(1|1,m)\right] \\
          &=  (1-\mu^{*\phi}_1)\left[p\pi(s|0)+ (1-p)\pi(m|0)\right] + \mu^{*\phi}_1\left[ (1-p)\pi(s|1)+p\pi(m|1)\right]\\
          &= (p\pi(s|0)+ (1-p)\pi(m|0)) - \mu^{*\phi}_1(p\pi(s|0)+ (1-p)\pi(m|0)) + \mu^{*\phi}_1((1-p)\pi(s|1)+p\pi(m|1))\\
          & = p\pi(s|0)+ (1-p)\pi(m|0)\\
          &= p\pi(s|0)+ (1-p)(1-\pi(s|0)) \\
          & = (1-p) -(1-2p)\pi(s|0)\\
          &=  (1-p) \left[\frac{\exp(-\phi\gamma (1-2p))}{1+\exp(-\phi\gamma (1-2p))}\right]+p\left[\frac{1}{1+\exp(-\phi\gamma (1-2p))}\right] \in (0,0.5)\quad \mbox{for}\quad p< 0.5 .
\end{align*}
It remains to prove that the unique $\mu^{*\phi}$ defined above is a GASE for $\mathcal{P}_2$.
Let us define the function
\[
g(x) = (1-p) \left[\frac{\exp(-\phi\gamma (1-2p))}{1+\exp(-\phi\gamma (1-2p))}\right]+p\left[\frac{1}{1+\exp(-\phi\gamma (1-2p))}\right] - x,
\]
so that $g(\mu^{*\phi}_1)=0$ for  $\mu^{*\phi}_1\in (0,0.5)$. We have $g(x)> 0$ for $0<x<\mu^{*\phi}_1$ and $g(x)<0$ for $\mu^{*\phi}_1<x<0.5$. Therefore, $L(x)=(x-\mu^{*\phi}_1)^2$ is 
a Lyapunov function for the ODE $\dot{y} = g(y)$ since $L'(y)g(y)<0$ and $L'(\mu^{*\phi}_1)g(\mu^{*\phi}_1)=0$. Repeating the argument with $\mu^{*\phi}_0=1-\mu^{*\phi}_1$, one deduces that $\mu^{*\phi}$ is the unique GASE in Assumption \ref{GASE}.

Regarding the gap in Theorem \ref{th:mainMFG}, we have:
\[\delta(\phi) = \min\left(Q^*_{\mu^{*\phi}}(0,m)-Q^*_{\mu^{*\phi}}(0,s),Q^*_{\mu^{*\phi}}(1,s)-Q^*_{\mu^{*\phi}}(1,m)\right)=\gamma (1-2p) = \delta>0,
\] for $p<0.5.$
\subsubsection{MFC}
As in Section \ref{sec:MFCformulation}, for a strategy $\alpha$, we define the strategy $\tilde{\alpha}^{(x,a)}$ by:
\begin{align*}
    \tilde{\alpha}^{(x,a)} (x') := \begin{cases}
        a\hspace{3mm} \text{if}\hspace{5mm} x'=x, \\
        \alpha(x')\hspace{2mm} \text{for}\hspace{3mm} x'\neq x.
    \end{cases} 
\end{align*}
Then, the MKV-Q-Bellman equation reads:
\[
Q^*(x,a)=f(x,a,\mu^{\tilde{\alpha^*}(x,a)})+\gamma\left[ (1-p)Q^*(a(x),s)\wedge Q^*(a(x),m) +p Q^*(1-a(x),s)\wedge Q^*(1-a(x),m)\right].
\]
One can check that the optimal strategy is $\alpha^*=(m,s)$ which corresponds to the case 
\[
Q^*(0,m)<Q^*(0,s),\quad Q^*(1,s)<Q^*(1,m),
\]
and to the distribution $\mu^*=(p,1-p)$.
The optimal Q-values are given by
\[
Q^*(0,m)=\frac{cp-\gamma p+\gamma}{1-\gamma}, \quad Q^*(1,s)=Q^*(0,m) +1,
\]
\[
Q^*(0,s)=\frac{c}{2}+\gamma\frac{cp-2\gamma p+\gamma +p}{1-\gamma},\quad Q^*(1,m)=Q^*(0,s)+1,
\]
so that $Q^*(0,s)-Q^*(0,m)=Q^*(1,m)-Q^*(1,s)=\frac{1}{2}(1-2p)(c-2\gamma)>0$ as soon as $c>2\gamma$.
Finally, note that $V^*(0)=Q^*(0,m)$ and $V^*(1)=Q^*(1,s)$.

\subsubsection{MFC Assumptions \ref{uniqueness minimum}, \ref{GASE-MFC}, and gap in Theorem \ref{th:mainMFC}} \label{verifMFC}
We prove in Appendix  \ref{MFCoptimal} that \ref{uniqueness minimum} is satisfied. Now, with the same notations as in  Appendix  \ref{MFCoptimal}, we use a $\softmin$  policy $\pi$ defined by $\pi(a|x)={\softmin}_\phi Q(x)(a)$.
For any given $Q$,  the four limiting distributions (denoted $\mu^{*\phi}_Q$ in Proposition \ref{emu_MFC}) are:
\[
\mu_1^{(0,s)} = \frac{p}{1-(1-2p)\pi(s|1)}, \quad \mu_1^{(0,m)} =\frac{1-p}{2-2p-(1-2p)\pi(s|1)},
\]
\[
\mu_1^{(1,s)} = \frac{1-p - (1-2p)\pi(s|0)}{1 - (1-2p)\pi(s|0)}, \quad \mu_1^{(1,m)} =\frac{1-p - (1-2p)\pi(s|0)}{2-2p - (1-2p)\pi(s|0)}.
\]
Then, we compute the GASE $Q^{*\phi}$ in Assumption \ref{GASE-MFC} by solving $\mathcal{T}_2(Q^{*\phi},\mu^{*\phi}_{Q^{*\phi}}) = 0$. By the calculation in Appendix \ref{MFCoptimal}, it is enough to compute under the case $Q^{*\phi}(0,s)>Q^{*\phi}(0,m)$ and $Q^{*\phi}(1,s)<Q^{*\phi}(1,m)$, since the other cases will not reach the GASE. That gives
\begin{align*}
Q^{*\phi}(0,m) &= \frac{c\mu_0^{^{*\phi}(0,m)}+\gamma(1-p)(1-c\mu_0^{^{*\phi}(0,m)} + c\mu_0^{^{*\phi}(1,s)})}{1-\gamma}\\ Q^{*\phi}(1,s)&=Q^{*\phi}(0,m)+1-c\mu_0^{^{*\phi}(0,m)} + c\mu_0^{^{*\phi}(1,s)}\\
Q^{*\phi}(0,s) &= c\mu_0^{^{*\phi}(0,s)}+\gamma\frac{c\mu_0^{^{*\phi}(0,m)}+\gamma(1-p)(1-c\mu_0^{^{*\phi}(0,m)} + c\mu_0^{^{*\phi}(1,s)})}{1-\gamma}+\gamma p(1-c\mu_0^{^{*\phi}(0,m)} + c\mu_0^{^{*\phi}(1,s)})\\
 Q^{*\phi}(1,m)&=Q^{*\phi}(0,s)+1-c\mu_0^{^{*\phi}(0,s)} + c\mu_0^{^{*\phi}(1,m)}
\end{align*}
We have four implicit equations for four unknowns (the $Q$'s), and one can check (by a tedous analysis) that there exists a unique solution.

It remains to prove that the unique $Q^{*\phi}$ defined above is a GASE for $\mathcal{T}_2$. Let us keep other entries for $Q$ the same as for $Q^{*\phi}$ and consider  first $Q(0,m)$. One can check that
\begin{align*}
    \begin{cases}
        \mathcal{T}_2(Q,\mu^{*\phi}_Q)(0,m) - Q(0,m) <0  \mbox{  if  } Q(0,m) > Q^{*\phi}(0,m),\\
        \mathcal{T}_2(Q,\mu^{*\phi}_Q)(0,m) - Q(0,m) >0  \mbox{  if  } Q(0,m) < Q^{*\phi}(0,m).
    \end{cases}
\end{align*}
The same is true for the other entries, therefore, $L(Q) = \|Q-Q^{*\phi}\|_\infty$ is a Lyapunov function for the ODE $\dot{Q} = \mathcal{T}_2(Q,\mu^{*\phi}_Q)$ since $L'(Q)\mathcal{T}_2(Q,\mu^{*\phi}_Q)<0$ and $L(Q^{*\phi})\mathcal{T}_2(Q^{*\phi},\mu^{*\phi}_{Q^{*\phi}}) = 0$. As a result, $Q^{*\phi}$ is the unique GASE in Assumption \ref{GASE-MFC}.

Regarding the gap in Theorem \ref{th:mainMFC}, we have:
\[\delta(\phi) = \min\left(Q^{*\phi}(0,m)-Q^{*\phi}(0,s),Q^{*\phi}(1,s)-Q^{*\phi}(1,m)\right)
\]
and $\lim_{\phi\to\infty} \delta(\phi) = Q^{*}(0,m)-Q^{*}(0,s) = Q^{*}(1,s)-Q^{*}(1,m) = \frac{1}{2}(1-2p)(c-2\gamma)>0$ as soon as $c>2\gamma$ and $p<0.5$.
\subsection{Numerical Results}\label{nr}

We consider the problem defined above by choice of parameters: $ N =2$ states, $ p=0.01$, $c=5$, $ \phi=500$, $ \gamma=0.5$. 
We present the following results for both MFG and MFC.

\subsubsection{MFG Experiments}
In order to test the stability of different learning rates, the solution of MFG problem is reached based on the 4 choices $(\omega^Q, \omega^{\mu}) = (0.55,0.85),(0.55,0.88),(0.55,0.92),(0.55,0.95)$. These choices satisfy the assumptions that we used to prove convergence. Figure~\ref{fig:example-mfg-omegas} illustrates the results.
We provide the convergence plot of $\mu_n(x_0)$, i.e., the value of the distribution at state 0, as a function of the step $n$. The y-axis is the value of $\mu_n(x_0)$, and x-axis is the number of iterations.

\begin{figure}[H]
\center 
\subfloat[MFG: $(\omega^Q, \omega^{\mu})=(0.55,0.85)$]
{\includegraphics[width=0.4\textwidth]{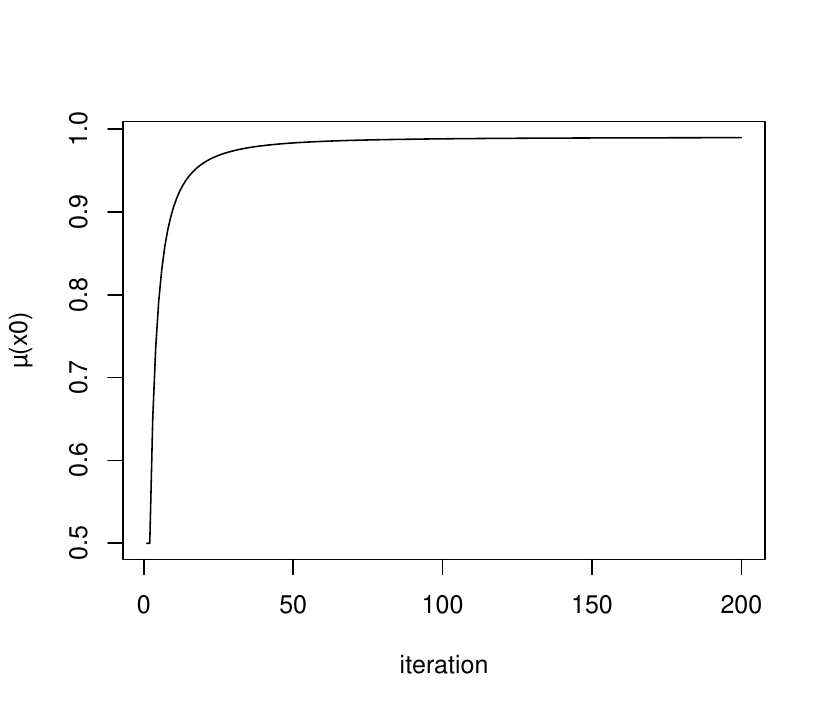}\label{fig:Ex_Im}}\qquad\qquad
\subfloat[MFG: $(\omega^Q, \omega^{\mu})=(0.55,0.88)$]{\includegraphics[width=0.4\textwidth]{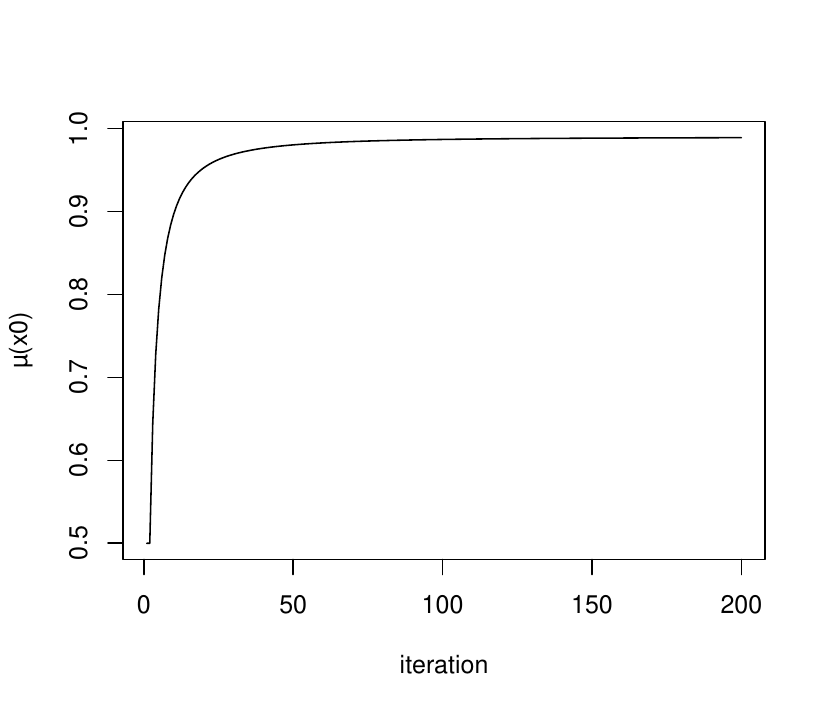}\label{fig:Ex_Im2}}

    \subfloat[MFG: $(\omega^Q, \omega^{\mu})=(0.55,0.92)$]{\includegraphics[width=0.4\textwidth]{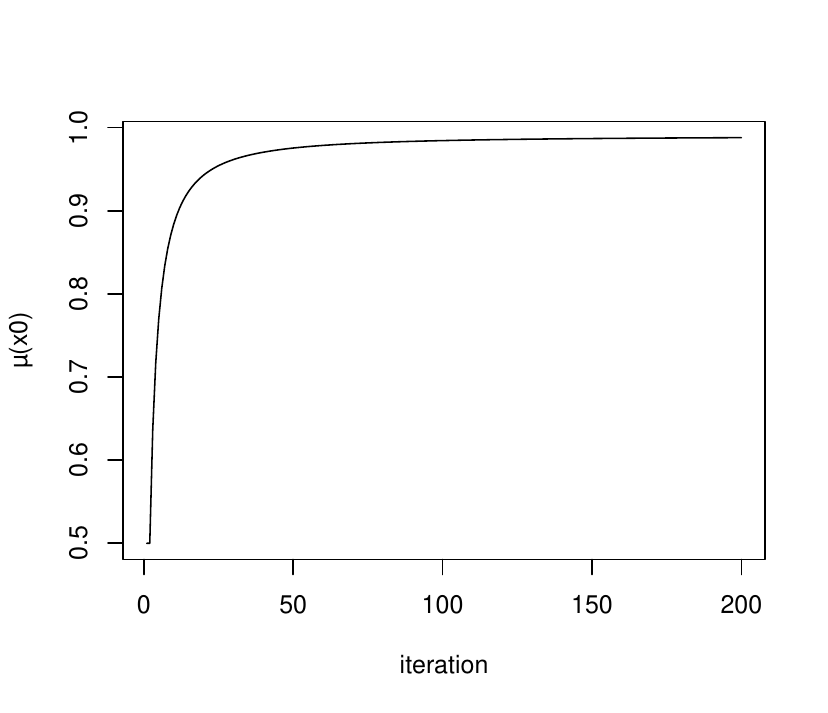}\label{fig:Ex_Im3}}\qquad\qquad
    \subfloat[MFG: $(\omega^Q, \omega^{\mu})=(0.55,0.95)$]{\includegraphics[width=0.4\textwidth]{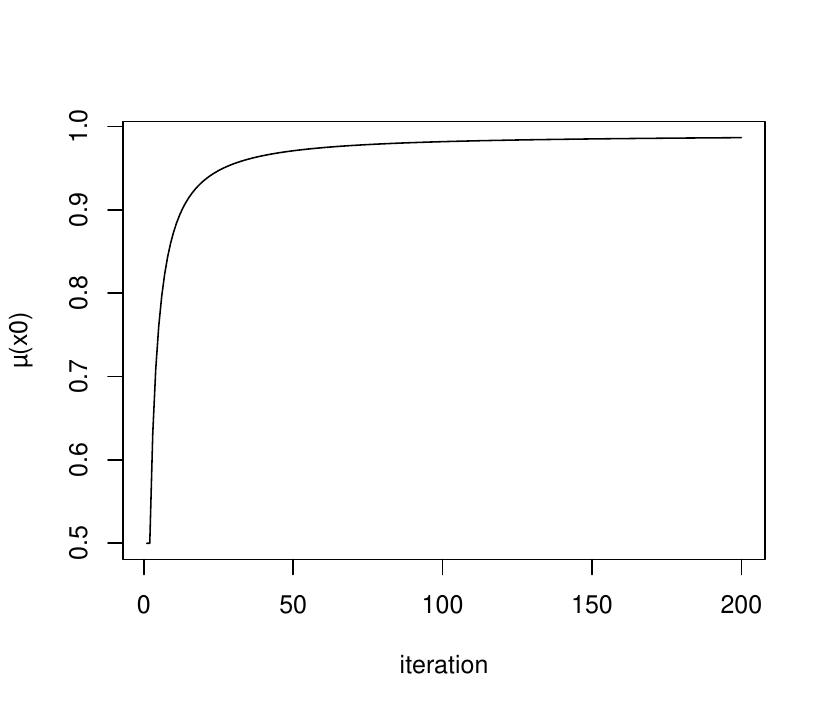}\label{fig:Ex_Im4}}
\caption{MFG setting: Convergence of the distribution. Each graph corresponds to one pair of $(\omega^Q,\omega^\mu)$, as indicated in the sub-caption. The plot represents the value of $\mu_n(x_0)$ as a function of $n$.   \label{fig:example-mfg-omegas}}
\end{figure}
The resulting distribution $\mu=(\mu(x_0),\mu(x_1))$ which is about  $(0.99,0.01)=(1-p,p)$, corresponds to the equilibrium distribution in the  MFG scenario with optimal strategy $\alpha^*=(s,m)$ given by the limiting Q-table:

\begin{center}
\begin{tabular}{ |p{3cm}||p{2cm}|p{2cm}|p{2cm}|p{2cm}|  }
 \hline
 & 
 \multicolumn{2}{|c|}{Q-table} 
 &                                            %
\multicolumn{2}{|c|}{Theoretical Values} \\
 \hline
 States/Actions & Action 0 &Action 1 & Action 0 & Action 1\\
 \hline
 State 0  &9.91 & 10.40  &9.91 &10.4\\
 \hline 
 State 1  & 11.40& 10.91 & 11.40& 10.91    \\

\hline
\end{tabular}
\end{center}

\subsubsection{MFC Experiments}
In order to test the stability of different learning rates, the solution of MFC problem is reached based on the 4 choices $(\omega^Q, \omega^{\mu}) = (0.85,0.55),(0.88,0.55),(0.92,0.55),(0.95,0.55)$. These choices satisfy the assumptions that we used to prove convergence. Figure~\ref{fig:example-mfc-omegas} illustrates the results.
We provide the convergence plot of $\mu^{\alpha^*}_n(x_0)$, i.e., the value of the distribution at state 0, as a function of the step $n$. The y-axis is the value of $\mu^{\alpha^*}_n(x_0)$, and x-axis is the number of iterations.

\begin{figure}[H]
\center 
\subfloat[MFC: $(\omega^Q, \omega^{\mu})=(0.85,0.55)$]{\includegraphics[width=0.4\textwidth]{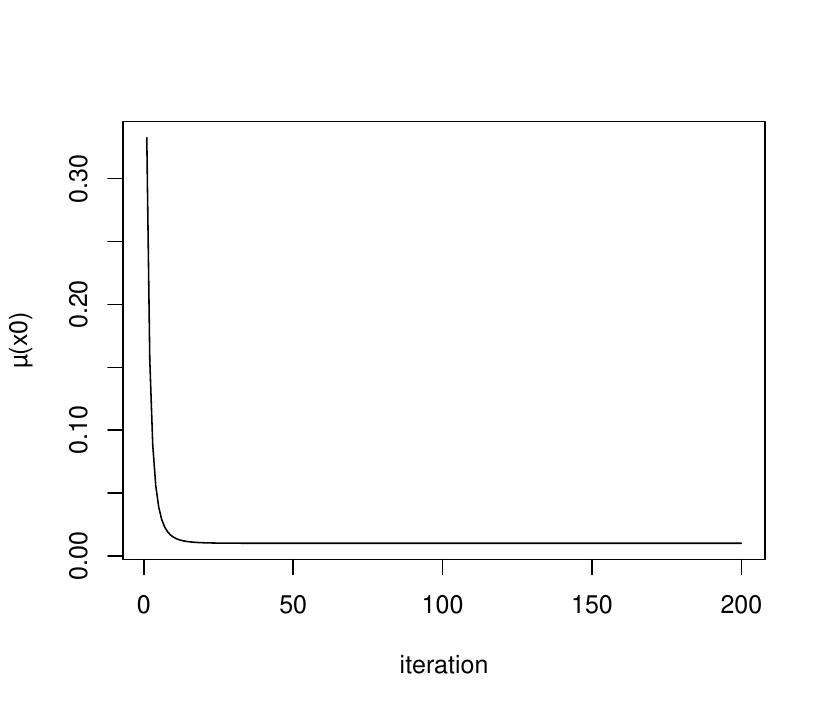}}\qquad\qquad
\subfloat[MFC: $(\omega^Q, \omega^{\mu})=(0.88,0.55)$]{\includegraphics[width=0.4\textwidth]{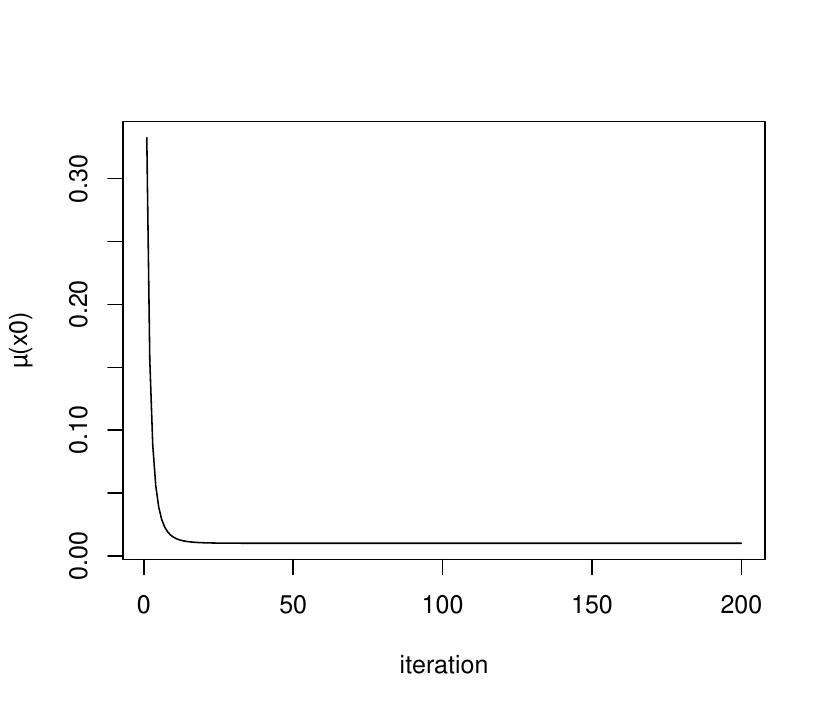}}

    \subfloat[MFC: $(\omega^Q, \omega^{\mu})=(0.92,0.55)$]{\includegraphics[width=0.4\textwidth]{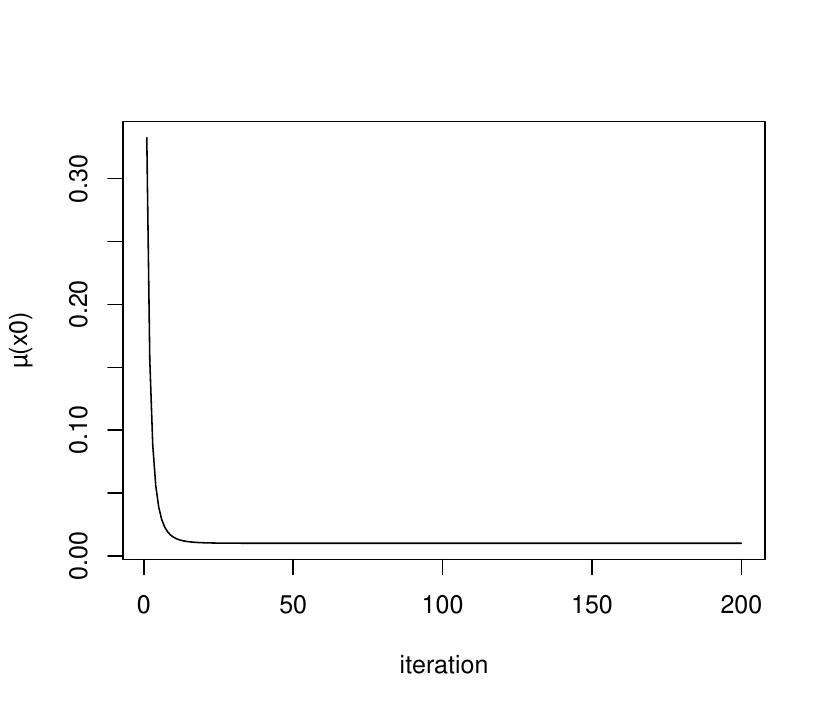}}\qquad\qquad
    \subfloat[MFC: $(\omega^Q, \omega^{\mu})=(0.95,0.55)$]{\includegraphics[width=0.4\textwidth]{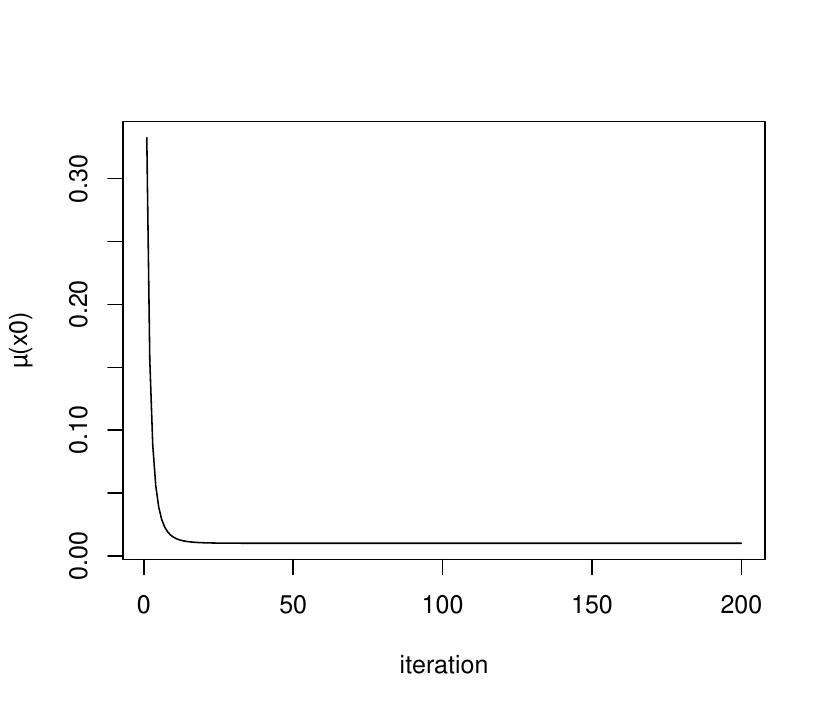}}

\caption{MFC setting: Convergence of the distribution. Each graph corresponds to one pair of $(\omega^Q,\omega^\mu)$, as indicated in the sub-caption. The plot represents the value of $\mu^{\alpha^*}_n(x_0)$ as a function of $n$. \label{fig:example-mfc-omegas}}
\end{figure}
The resulting distribution $\mu=(\mu(x_0),\mu(x_1))$ which is about  $(0.01,0.99)=(p,1-p)$ corresponds to the equilibrium distribution in the  MFC scenario with optimal strategy $\alpha^*=(m,s)$ given by the limiting Q-table:

\begin{center}
\begin{tabular}{ |p{3cm}||p{2cm}|p{2cm}|p{2cm}|p{2cm}|  }
 \hline
 & 
 \multicolumn{2}{|c|}{Q-table }
  &                                            %
\multicolumn{2}{|c|}{Theoretical Values} \\
 \hline
 States/Actions& Action 0 &Action 1& Action 0 &Action 1\\
 \hline
 State 0 &3.052254 & 1.090001 &3.05 &1.09\\
 \hline
 State 1 & 2.090001 & 4.050001 & 2.09 & 4.05\\
 \hline
\end{tabular}
\end{center}

\subsubsection{Robustness with respect to the learning rates}

In this section, we  show the robustness with respect to the choice of learning rate parameters $\omega^Q$ and $\omega^\mu$. In order to show the robustness, we choose $(\omega^Q, \omega^{\mu})$ among \{0.55, 0.6, 0.65, 0.7, 0.75, 0.8, 0.85, 0.9, 0.95\}. 
In Figure \ref{fig:ode-omegas}, we provide the surface plot of the limiting value of $\mu(x_0)$. The z-axis is the value of $\mu(x_0)$, the x-axis and y-axis are $\omega^Q$ and $\omega^\mu$ respectively. In this plot, we force $\mu(x_0) = 0.5$ if $\omega^Q=\omega^\mu$, since it is not part of the region we want to analyse.

\begin{figure}[H]
\center 
\subfloat{\includegraphics[width=0.8\textwidth,height=6cm]{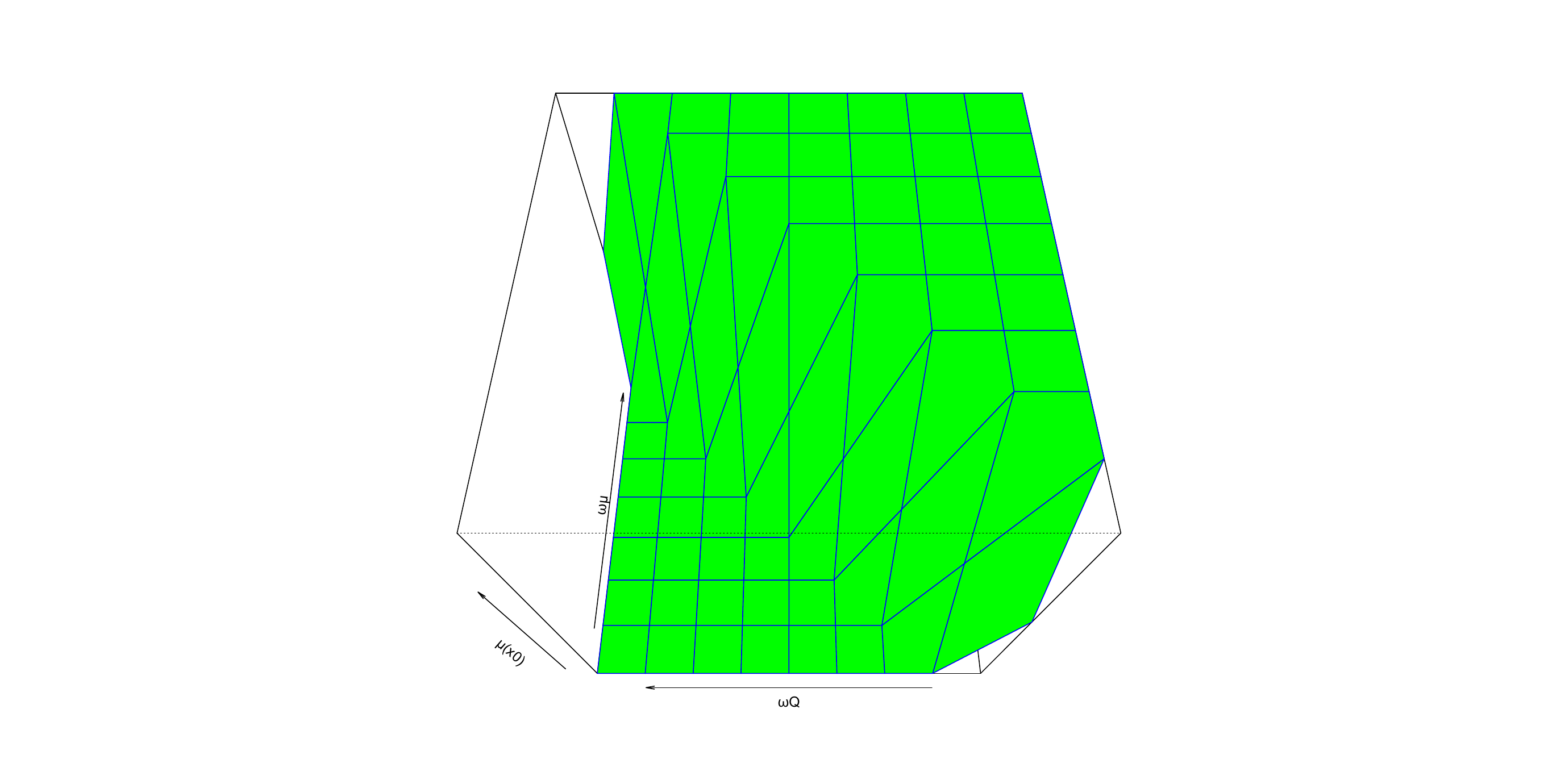}\label{fig:ODEsurface}}
\caption{Robustness of the distribution with respect to the different choices of learning rate parameters $\omega^Q$ and $\omega^\mu$.
For MFG $\omega^\mu>\omega^Q$ and for MFC $\omega^\mu<\omega^Q$.
\label{fig:ode-omegas}}
\end{figure}
 From this plot, one observes the two regimes: MFG and MFC. Within these two regimes, our idealized deterministic two-timescale iteration systems \eqref{simpletwotimescaleapproach_mfg} and \eqref{simpletwotimescaleapproach_mfc} can accurately learn the result  with respect to different choices of learning rate parameters $\omega^Q$ and $\omega^\mu$.

\section{Stochastic Approximation}
\label{sec:sto-approx}
In the previous section, we introduced an idealized deterministic algorithm relying on the operators $\mathcal{P}_2$ and $\mathcal{T}_2$, which involve expectations that we assumed we could compute perfectly. However, in practical situations, those expectations are unknown and instead we use samples. Then we can use the classical  theory of stochastic approximation to obtain convergence of algorithms with sample-based updates. In this section, we assume that the updates are still done in a synchronous way. The case of asynchronous updates is discussed in the next section.

In this section, we consider the following setting. Let us assume that for any $(\mu,x,a)$, the learner can know the value $f(x,a,\mu)$. Furthermore, the learner can sample a realization of the random variable
\begin{align*}
    X^{'}_{x,a,\mu}\sim p(\cdot|x,a,\mu),
\end{align*}
Then, the learner has access to realizations of the following random variables $\widecheck{\mathcal{T}}_{\mu,Q}(x,a)$ and $\widecheck{\mathcal{P}}_{x,a,\mu}(\mu)$ taking values respectively in $\mathbbm{R}$ and $\Delta^{|\mathcal{X}|}$:
\begin{align*}
    \begin{cases}
        \widecheck{\mathcal{T}}_{\mu,Q}(x,a) = f(x,a,\mu)+\gamma\min_{a'}Q(X^{'}_{x,a,\mu},a')-Q(x,a)\\
        (\widecheck{\mathcal{P}}_{x,a}(\mu)(x'))_{x'\in\mathcal{X}} = \left(\mathbbm{1}_{\{X^{'}_{x,a,\mu}=x'\}}-\mu(x')\right)_{x'\in\mathcal{X}}.
    \end{cases}
\end{align*}
Observe that 
\begin{align*}
    \begin{cases}
        \mathbb{E}[\widecheck{\mathcal{T}}_{\mu,Q}(x,a)] = \sum_{x''}p(x''|x,a,\mu)\left[f(x,a,\mu)+\gamma\min_{a'}Q(x'',a')-Q(x,a)\right] = \mathcal{T}_2(Q,\mu)(x,a)\\
        \mathbb{E}[\widecheck{\mathcal{P}}_{x,a}(\mu)(x')]=\sum_{x''}p(x''|x,a,\mu)(\mathbbm{1}_{\{x''=x'\}}-\mu(x')) = p(x'|x,a,\mu) - \mu(x').
    \end{cases}
\end{align*}

For MFG, if the starting point $x$ comes from a random variable $X\sim\mu$ and if $a$ is an action chosen from a soft-min policy at $X$ according to $Q$, i.e., with probability distribution $\softmin Q(X,\cdot)$, then we obtain
\begin{align}
    \mathbb{E}[\widecheck{\mathcal 
 {P}}_{X,\softmin Q(X)}(\mu)(x')] &=\sum_{x}\mu(x)\sum_{x''}p(x''|x,\softmin Q(x),\mu)(\mathbbm{1}_{\{x''=x'\}}-\mu(x'))\nonumber\\
    &=\mathcal{P}_2(Q,\mu)(x'). \label{eq:exp-Pcheck-P}
\end{align}

Because of that, for MFG we can replace the deterministic updates by the following stochastic ones, starting from some initial $Q_0$, $\mu_0$: for $n=0,1,...,$
\begin{align}
    \begin{cases}\label{stochastic_two}
        \mu_{n+1}(x)&= \mu_n(x) + \rho^{\mu}_n\widecheck{\mathcal{P}}_{X_n,\softmin Q(X_n)}(\mu_n)(x) \\
        &=\mu_n(x) + \rho^{\mu}_n (\mathcal{P}_2(Q_n,\mu_n)(x)+\mathbf{P}_n(x)),\hspace{10pt}\forall x \in \mathcal{X}\\
        Q_{n+1}(x,a)&= Q_n(x,a)+\rho^Q_n\widecheck{\mathcal{T}}_{\mu_n,Q_n}(x,a)\\
        &= Q_n(x,a) + \rho^Q_n(\mathcal{T}_2(Q_n,\mu_n)(x,a)+\mathbf{T}_n(x,a)),\hspace{10pt}\forall(x,a)\in\mathcal{X}\times\mathcal{A} \\
        \hspace{30pt}X_n&\sim \mu_n
    \end{cases}
\end{align}
where we introduced the notation:
\begin{align}
\label{MDs}
    \begin{cases}
        \mathbf{P}_n(x) = \widecheck{\mathcal{P}}_{X_n,\softmin Q(X_n)}(\mu_n)(x) - \mathcal{P}_2(Q_n,\mu_n)(x)\\
        \mathbf{T}_n(x,a) = \widecheck{\mathcal{T}}_{\mu_n,Q_n}(x,a) - \mathcal{T}_2(Q_n,\mu_n)(x,a)
    \end{cases}
\end{align}
with $X_n$ sampled from $\mu_n$. Note that $\mathbf{T}_n$ and $\mathbf{P}_n$ are martingale difference sequences.\footnote{By Markov property, $\mathbb{E}[\mathbf{P}_n(x)|\mathcal{F}_{n-1}] = \mathbb{E}[\widecheck{\mathcal{P}}_{X_n,\argmin_{a}Q(X_n,a),\mu_n}(\mu_n)(x)|\mathcal{F}_{n-1}] - \mathcal{P}_3(\mu_n,Q_n,\mu_n)(x) = \mathbb{E}[\widecheck{\mathcal{P}}_{X_n,\argmin_{a}Q(X_n,a),\mu_n}(\mu_n)(x)|X_{n-1}=x_{n-1}] - \mathcal{P}_3(\mu_n,Q_n,\mu_n)(x) = 0$. Similarly we have that $\mathbf{T}_n$ is also martingale difference sequence.\label{ft}}\\

Similarly for MFC, if the starting point $x'$ comes from a random variable $X^{(x,a)}\sim\mu^{(x,a)}$ and if $a'$ is an action chosen from the following policy,
\begin{align}
    \tilde\pi(x') := \begin{cases}
        \delta_a\hspace{3mm} \text{if}\hspace{5mm} x'=x, \\
        \softmin Q(X^{(x,a)},\cdot)\quad \text{for}\hspace{3mm} x'\neq x.
    \end{cases} 
\end{align}
then we obtain
\begin{align}
    \mathbb{E}[\widecheck{\mathcal 
 {P}}_{X,\tilde\pi(X)}(\mu^{(x,a)})(y)] &=\sum_{x'\neq x}\mu^{(x,a)}(x')\sum_{x''}p(x''|x',\softmin Q(x'),\mu^{(x,a)})(\mathbbm{1}_{\{x''=y\}}-\mu^{(x,a)}(y))\nonumber\\
    &\quad + \mu^{(x,a)}(x)\sum_{x''}p(x''|x,a,\mu^{(x,a)})(\mathbbm{1}_{\{x''=y\}}-\mu^{(x,a)}(y))\nonumber\\
    &=\sum_{x'\neq x}\mu^{(x,a)}(x')p(y|x',{\softmin}_\phi Q(x'),\mu^{(x,a)}) +\mu^{(x,a)}(x)p(y|x,a,\mu^{(x,a)}) - \mu^{(x,a)}(y)\nonumber\\
    &=\mathcal{\tilde P}^{(x,a)}_2(Q,\mu^{(x,a)})(y). \label{eq:mfc-Pcheck-P}
\end{align}

Because of that, for MFC we can replace the deterministic updates by the following stochastic ones, starting from some initial $Q_0$, $\mu^{(x,a)}_0$: for $n=0,1,...,$
\begin{align}
    \begin{cases}\label{stochastic_two_MFC}
        \mu^{(x,a)}_{n+1}(y)&= \mu^{(x,a)}_n(y) + \rho^{\mu}_n\widecheck{\mathcal{P}}_{X^{(x,a)}_n,\tilde\pi(X^{(x,a)}_n)}(\mu^{(x,a)}_n)(y) \\
        &=\mu^{(x,a)}_n(y) + \rho^{\mu}_n (\mathcal{\tilde P}^{(x,a)}_2(Q_n,\mu^{(x,a)}_n)(y)+\mathbf{P}^{(x,a)}_n(y)),\hspace{10pt}\forall y \in \mathcal{X}\\
        Q_{n+1}(x,a)&= Q_n(x,a)+\rho^Q_n\widecheck{\mathcal{T}}_{\mu^{(x,a)}_n,Q_n}(x,a)\\
        &= Q_n(x,a) + \rho^Q_n(\mathcal{T}_2(Q_n,\mu^{(x,a)}_n)(x,a)+\mathbf{T}_n(x,a)),\hspace{10pt}\forall(x,a)\in\mathcal{X}\times\mathcal{A} \\
        \hspace{30pt}X^{(x,a)}_n&\sim \mu^{(x,a)}_n
    \end{cases}
\end{align}
where we introduced the notation:
\begin{align}
\label{MDs_MFC}
    \begin{cases}
        \mathbf{P}^{(x,a)}_n(y) = \widecheck{\mathcal{P}}_{X^{(x,a)}_n,\tilde\pi(X^{(x,a)}_n)}(\mu^{(x,a)}_n)(y) - \mathcal{\tilde P}^{(x,a)}_2(Q_n,\mu^{(x,a)}_n)(y)\\
        \mathbf{T}_n(x,a) = \widecheck{\mathcal{T}}_{\mu^{(x,a)}_n,Q_n}(x,a) - \mathcal{T}_2(Q_n,\mu^{(x,a)}_n)(x,a)
    \end{cases}
\end{align}
with $X^{(x,a)}_n$ sampled from $\mu^{(x,a)}_n$. Now we clearly define the system of iteration referring to \cite{Borkar97}.

Next, we will establish the convergence of the two-timescale approach with stochastic approximation defined in \eqref{stochastic_two} and \eqref{stochastic_two_MFC}.
\begin{proposition}\label{squaremartingale}
    Suppose  Assumptions~\ref{squaresumlearningrate_mfg} and Assumption~\ref{squaresumlearningrate_mfc} hold. Let $\psi^\mu_n(x):=\sum_{m=1}^n \rho^\mu_m\mathbf{P}_m(x)$, $\psi^{\mu^{(x,a)}}_n(y):=\sum_{m=1}^n \rho^\mu_m\mathbf{P}^{(x,a)}_m(y)$ and $\psi^Q_n(x,a):=\sum_{m=1}^n \rho^Q_m \mathbf{T}_m(x,a)$. They are square integrable martingales and hence they converge a.s. as $n \to +\infty$.
\end{proposition}
\begin{proof}
    By Lemma 4.5 in \cite{Borkar00}, we directly have $\psi^Q_n$ is a square integrable martingale. Also we have $E[\|\mathbf{{P}}_n\|^2]<1$ and $E[\|\mathbf{{P}}^{(x,a)}_n\|^2]<1$. Using this and the square summability of $\{\rho^{\mu}_n\}$ assumed in Assumption~\ref{squaresumlearningrate_mfg} and Assumption~\ref{squaresumlearningrate_mfc}, the bound immediately follows, which shows that $\psi^\mu_n$ and $\psi^{\mu^{(x,a)}}_n$ are also square integrable martingales. Then by martingale convergence theorem \cite[p.62]{neveu1975discrete}, we have the convergence of the three martingales.
\end{proof}

For MFG, we have the following theorem. 
\begin{theorem}\label{Stocastic MFG}
    Suppose Assumptions \ref{squaresumlearningrate_mfg}, \ref{fp_lipschitz}, \ref{mfclp} and \ref{GASE} hold. Then $(\mu_n, Q_n)_{n \ge 0}$ defined in \eqref{stochastic_two} converges to $(\mu^{*\phi},Q^*_{\mu^{*\phi}})$ a.s. as $n \to\infty$.
\end{theorem}
\begin{proof}
    With Assumption \ref{squaresumlearningrate_mfg}, \ref{fp_lipschitz}, \ref{mfclp} and \ref{GASE}, and the results of Propositions \ref{PLipschitz}, \ref{TLipschitz},  and \ref{squaremartingale}, the assumptions of \cite[Theorem 1.1]{Borkar97} are satisfied. This result guarantees the convergence in the statement.
\end{proof}
For MFC, 
we define $V_n(x)=\min_a Q_n(x,a)$ where $Q_n$ satisfies \eqref{stochastic_two_MFC}. Then,
we have: 
\begin{theorem}\label{Stocastic MFC}
    Suppose Assumptions  \ref{fp_lipschitz}, \ref{mfclp}, \ref{squaresumlearningrate_mfc}, and \ref{GASE-MFC} hold. Then $(\mu_n, V_n)_{n \ge 0}$  converges to $(\mu^{*\phi}_{Q^{*\phi}},V^{*\phi})$ a.s. as $n \to\infty$.
\end{theorem}
\begin{proof}
    With Assumptions \ref{fp_lipschitz}, \ref{mfclp}, \ref{squaresumlearningrate_mfc},  \ref{GASE-MFC} and the results of Propositions \ref{PLipschitz}, \ref{TLipschitz},  and \ref{squaremartingale}, the assumptions of \cite[Theorem 1.1]{Borkar97} are satisfied. This result guarantees the convergence in the statement.
\end{proof}

\section{Asynchronous Setting}\label{sec:sto-approx-asynchronous}
In the previous section, we assumed that learner have access to a generative model, i.e., to a simulator which can provide the samples of transitions drawn according to the hidden dynamic for arbitrary state $x$. However, in a more realistic situation, the learner is constrained to follow the trajectory sampled by the environment without the ability to choose arbitrarily its state. We call this an \emph{asynchronous setting} and we define the system as follows for MFG,
    \begin{subnumcases}{}
    \label{diffa}
        \mu_{n+1} = \mu_n + \rho_{n,X_n,A_n}^{\mu}(\mathcal{P}_2(Q_n,\mu_n) + \mathbf{P}_n)\\
        \label{diffb}
        Q_{n+1}(X_n,A_n) = Q_n(X_n,A_n) + \rho^Q_{n,X_n,A_n}({\mathcal{T}_2}(Q_n,{\mu}_n)(X_n,A_n)+\mathbf{T}_n(X_n,A_n))
    \end{subnumcases}
for $x\in\mathcal{X},\:a\in\mathcal{A}, n\geq0$, where $(X_n,A_n)$ is the state-action pair at time $n$. Comparing to the synchronous environment in the previous section, now we update the $Q$-table at one state-action pair for each time. As a consequence, the state-action pairs are not all visited at the same frequency and the learning rate needs to be adjusted accordingly. We denote the learning rate for each state-action pair $(x,a)$ as $\rho^Q_{n,x,a}$. The martingale difference sequence is defined as those in the previous section, namely~\eqref{MDs}.

In turns, for MFC, we define the following system,
    \begin{subnumcases}{}
    \label{diffc}
        \mu^{(x,a)}_{n+1} = \mu^{(x,a)}_n + \rho_{n,X^{(x,a)}_n,A^{(x,a)}_n}^{\mu}(\mathcal{\tilde P}_2(Q_n,\mu^{(x,a)}_n) + \mathbf{P}^{(x,a)}_n)\\
        \label{diffd}
        Q_{n+1}(x,a) = Q_n(x,a) + \rho^Q_{n,X^{(x,a)}_n,A^{(x,a)}_n}({\mathcal{T}_2}(Q_n,{\mu}^{(x,a)}_n)(x,a)+\mathbf{T}_n(x,a))\mathbbm{1}_{\{X^{(x,a)}_n=x\}}
    \end{subnumcases}
    for $x\in\mathcal{X},\:a\in\mathcal{A}, n\geq0$, where $(X^{(x,a)}_n,A^{(x,a)}_n)$ is the state-action pair at time $n$. For all those paths, we denote the learning rate for each state-action pair $(x,a)$ as $\rho^Q_{n,x,a}$. The martingale difference sequence is defined  as those in the previous section, namely~\eqref{MDs_MFC}.

Next, we will establish the convergence of the two-timescale approach with stochastic approximation under asynchronous setting defined in \eqref{diffa}--\eqref{diffb}.

We will use the following assumption. 
\begin{assumption}\label{dlearningrate}
    For any $(x,a)$, the sequence  $(\rho^{\mu}_{n,x,a})_{n\in \mathbbm{N}}$ satisfies the following conditions: 
     \[\sum_k \rho^{\mu}_{k,x,a} = \infty\quad \mbox{and} \quad \sum_k |\rho^{\mu}_{k,x,a}|^2<\infty.\]
\end{assumption}
\begin{assumption}[Ideal tapering stepsize (ITS)]\label{asynchronouslearningrate}
    For any $(x,a)$, the sequence  $(\rho^Q_{n,x,a})_{n\in \mathbbm{N}}$ satisfies the following conditions: 
    \begin{itemize}
        \item (i) $\sum_k \rho^Q_{k,x,a} = \infty,$ and  $\sum_k |\rho^Q_{k,x,a}|^2<\infty.$
        \item (ii) $\rho^Q_{k+1,x,a}\leq \rho^Q_{k,x,a}$ from some k onwards.
        \item (iii) there exists $r\in(0,1)$ such that
        $
            \sum_{k}(\rho^Q_{k,x,a})^{(1+q)}<\infty,\hspace{20pt}\forall q\geq r.
        $
        \item (iv) for $\xi\in(0,1)$, 
        $
            \sup_{k} \frac{\rho^Q_{[\xi k],x,a}}{\rho^Q_{k,x,a}}<\infty,
        $
        where $[\cdot]$ stands for the integer part.
        \item (v) for $\xi \in(0,1)$, and $D(k):=\sum_{\ell=0}^k \rho^Q_{\ell,x,a}$, 
        $
            \frac{D([yk])}{D(k)}\to 1,
        $
        uniformly in $y\in[\xi,1]$.
    \end{itemize}
\end{assumption}

In the asynchronous setting, we update step-by-step the Q table along the trajectory generated randomly. In order to ensure convergence, every state-action pair needs to be updated repeatedly. We introduce $\nu(x,a,n)=\sum_{m=0}^{n}\mathbbm{1}_{\{(X_m,A_m)=(x,a)\}}$ as the number of times the process $(X_n,A_n)$ visit state $(x,a)$ up to time $n$. 
The following assumptions are related to the (almost sure) good behavior of the number visits along trajectories. 
\begin{assumption}\label{frequentupdate}
    There exists a deterministic $\Delta > 0$ such that for all $(x,a)$, 
\begin{align*}
    \liminf_{n\to\infty} \frac{\nu(x,a,n)}{n}\geq\Delta\:\:\:\:\:a.s.
\end{align*}
Furthermore, letting
$
    N(n,\xi) = \min\left\{m>n:\sum_{k=n+1}^{m}\rho^Q_{k,x,a}>\xi\right\} 
$
for $\xi>0$, the limit 
\begin{align*}
    \lim_{n\to\infty}\frac{\sum_{k=\nu((x,a),n)}^{\nu((x,a),N(n,\xi))}\rho^Q_{k,x,a}}{\sum_{k=\nu((x',a'),n)}^{\nu((x',a'),N(n,\xi))}\rho^Q_{k,x',a'}}
\end{align*}
exists a.s. for all pairs $(x,a)$, $(x',a')$. 
\end{assumption}

\begin{assumption}\label{balanced}
    There exists $d_{ij}>0$, such that for each $i=(x,a)$ and $j=(x',a')$
    \begin{align*}
        \lim_{n\to\infty} \frac{\sum_{m=0}^n \rho^Q_{m,j}}{\sum_{m=0}^n \rho^Q_{m,i}}=d_{ij}.
    \end{align*}
\end{assumption}

\begin{proposition}\label{lyp}
    Suppose Assumptions~\ref{fp_lipschitz}, \ref{mfclp}, and~\ref{GASE} hold. %
    Then for every $\mu$, there exists a strict Lyapunov function for the function $Q \mapsto \mathcal{T}_2(Q,\mu)$. Likewise, there exists a strict Lyapunov function for the function $Q \mapsto \mathcal{T}_2(Q,{\mu^{*\phi}_{Q}})$.
\end{proposition}
\begin{proof}
    
    This is a direct consequence of the strict contraction property. For the first part, let us fix $\mu$. Recall from  that $Q \mapsto \mathcal{B}_\mu Q$ is a strict contraction over $Q$ for any given $\mu$. 
    Let $\tilde{Q}$ be defined by  $\mathcal{T}_2(\tilde{Q},\mu) = 0$. 
    Let us define $V= \|Q-\tilde{Q}\|_\infty$. By \cite[Case 1, Page 844]{Borkar98}, $V$ is the strict Lyapunov function for $Q \mapsto \mathcal{T}_2(Q,\mu)$. For the second part, recall that $Q \mapsto \mathcal{B}_{\mu^{*\phi}_Q}Q$ is a strict contraction. %
    So we can define $\tilde{Q}$ such that $\mathcal{T}_2(\tilde{Q},{\mu^{*\phi}_{\tilde{Q}}}) = 0$, and by \cite[Case 1, Page 844]{Borkar98} $V= \|Q-\tilde{Q}\|_\infty$ is the corresponding strict Lyapunov function.
\end{proof}

For MFG, we have the following theorem.
\begin{theorem}\label{Asyc MFG}
    Suppose Assumptions~\ref{fp_lipschitz}, \ref{mfclp}, 
    \ref{GASE}, %
    \ref{squaresumlearningrate_mfg}, 
    \ref{dlearningrate}, \ref{asynchronouslearningrate}, \ref{frequentupdate} and \ref{balanced} hold. Then, $(\mu_n, Q_n)_{n \ge 0}$ defined in \eqref{simpletwotimescaleapproach_mfg} converges to $(\mu^{*\phi},Q^*_{\mu^{*\phi}})$ a.s. as $n \to\infty$.
\end{theorem}
\begin{proof}
    With the stated assumptions and the results of Propositions \ref{PLipschitz}, \ref{TLipschitz}, 
    \ref{eQ_MFG}, 
    \ref{squaremartingale} and \ref{lyp}, the assumptions of \cite[Theorem 3.1]{Borkar98} and \cite[Lemma 4.11]{Konda99} are satisfied. These results guarantee the convergence in the statement.
\end{proof}

For MFC, we have the following theorem. 
\begin{theorem}\label{Asyc MFC}
    Suppose Assumptions \ref{fp_lipschitz}, \ref{mfclp}, 
    \ref{GASE-MFC}, %
    \ref{squaresumlearningrate_mfc}, 
    \ref{dlearningrate}, \ref{asynchronouslearningrate}, \ref{frequentupdate} and \ref{balanced} hold. Then, $(\mu_n, V_n)$ defined in Theorem \ref{thMFCcvg} converges to $(\mu^{*\phi}_{Q^{*\phi}},V^{*\phi})$ a.s. as $n \to\infty$.
\end{theorem}
\begin{proof}
    With the stated assumptions and the results of Propositions \ref{PLipschitz}, \ref{TLipschitz}, \ref{emu_MFC}, 
    \ref{squaremartingale} and \ref{lyp}, the assumptions of \cite[Theorem 3.1]{Borkar98} and \cite[Lemma 4.11]{Konda99} are satisfied. These results guarantee the convergence in the statement.
\end{proof}

\section{Numerical Experiments with Algorithms \ref{algo:U2MFQL} and \ref{algo:U2MFQL-MFC}}
We run the algorithms \ref{algo:U2MFQL} (MFG) and \ref{algo:U2MFQL-MFC}
(MFC) for the example given in the Section \ref{nr}. 

In order to test the stability of Algorithms \ref{algo:U2MFQL}, the solution of the MFG problem is reached based on the choice $(\omega^Q, \omega^{\mu}) = (0.55,0.85)$. This choice satisfies the assumptions that we used to prove convergence. Figure~\ref{fullmfg} illustrates the result.
We provide the convergence plot of $\mu_n(x_0)$, i.e., the value of the distribution at state 0, as a function of the step $n\leq 10000$. The y-axis is the value of $\mu_n(x_0)$, and x-axis is the number of iterations.  

\begin{figure}[H]
\center 
\subfloat[MFG: $(\omega^Q, \omega^{\mu})=(0.55,0.88)$]{\includegraphics[width=0.6\textwidth,height=6cm]{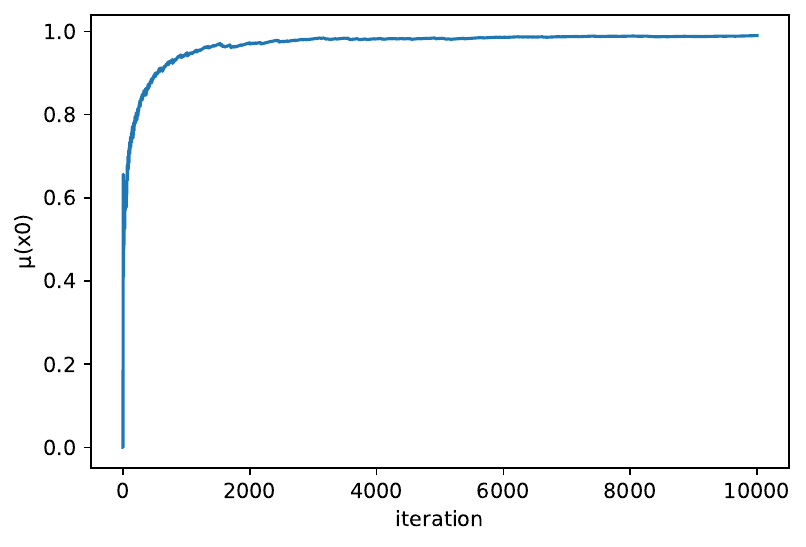}}
\caption{MFG setting: Convergence of the distribution. The plot represents the value of $\mu_n(x_0)$ as a function of $n$.}
\label{fullmfg}
\end{figure}
The resulting distribution $\mu=(\mu(x_0),\mu(x_1))$ which is about  $(0.99,0.01)=(1-p,p)$, corresponds to the equilibrium distribution in the  MFG scenario with optimal strategy $\alpha^*=(s,m)$ given by the limiting Q-table:

\begin{center}
\begin{tabular}{ |p{3cm}||p{2cm}|p{2cm}|p{2cm}|p{2cm}|  }
 \hline
 & 
 \multicolumn{2}{|c|}{Q-table} 
 &                                            %
\multicolumn{2}{|c|}{Theoretical Values} \\
 \hline
 States/Actions & Action 0 &Action 1 & Action 0 & Action 1\\
 \hline
 State 0  &9.841 & 10.051  &9.91 &10.4\\
 \hline 
 State 1  & 11.061& 10.761 & 11.40& 10.91    \\

\hline
\end{tabular}
\end{center}

In order to test the stability of Algorithms \ref{algo:U2MFQL-MFC}, the solution of MFC problem is reached based on the choice $(\omega^Q, \omega^{\mu}) = (0.85,0.55)$. This choice satisfies the assumptions that we used to prove convergence. Figure~\ref{fullmfc} illustrates the results.
We provide the convergence plot of $\mu^{\alpha^*}_n(x_0)$, i.e., the value of the distribution at state 0, as a function of the step $n\leq 100000$. The y-axis is the value of $\mu^{\alpha^*}_n(x_0)$, and x-axis is the number of iterations.
As expected, there are more fluctuations in the MFC case than in the MFG case. Indeed, in the MFC case the distribution is updated faster.
 \begin{figure}[H]
\center 
\subfloat[MFC: $(\omega^Q, \omega^{\mu})=(0.85,0.55)$]{\includegraphics[width=0.6\textwidth,height=6cm]{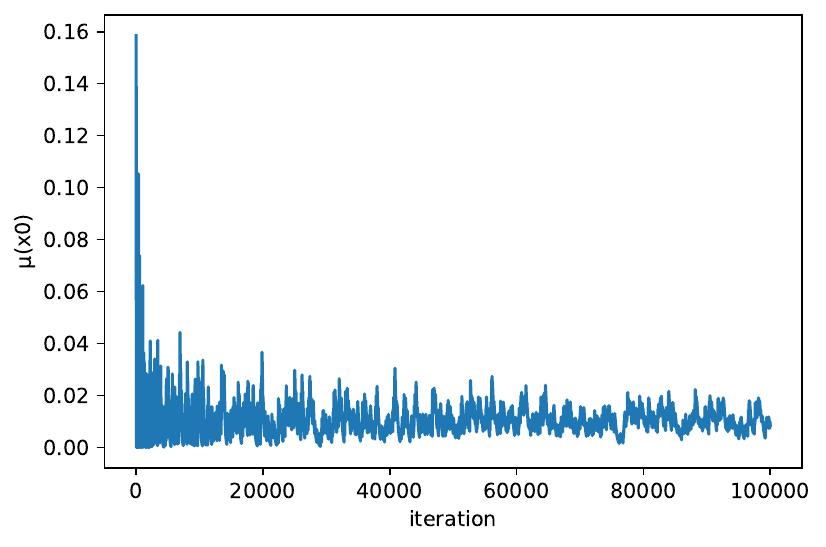}}
\caption{MFC setting: Convergence of the distribution. The plot represents the value of $\mu^{\alpha^*}_n(x_0)$ as a function of $n$.}
\label{fullmfc}
\end{figure}
The resulting distribution $\mu=(\mu(x_0),\mu(x_1))$ which is about  $(0.01,0.99)=(p,1-p)$ corresponds to the equilibrium distribution in the  MFC scenario with optimal strategy $\alpha^*=(m,s)$ given by the limiting Q-table:

\begin{center}
\begin{tabular}{ |p{3cm}||p{2cm}|p{2cm}|p{2cm}|p{2cm}|  }
 \hline
 & 
 \multicolumn{2}{|c|}{Q-table }
  &                                            %
\multicolumn{2}{|c|}{Theoretical Values} \\
 \hline
 States/Actions& Action 0 &Action 1& Action 0 &Action 1\\
 \hline
 State 0 &3.105 & 1.112 &3.05 &1.09\\
 \hline
 State 1 & 2.104 & 4.063 & 2.09 & 4.05\\
 \hline
\end{tabular}
\end{center}
Figure \ref{fig:A1surface} and \ref{fig:A1error} illustrate the robustness of our Algorithm \ref{algo:U2MFQL} and \ref{algo:U2MFQL-MFC} respect to the the choice of learning rate parameters $\omega^Q$ and $\omega^\mu$. For MFG regime, we use one run for 500000 steps, and for MFC regime, we use 100 runs for 100000 steps and display the average.

 \begin{figure}[H]
\center 
\subfloat{\includegraphics[width=0.4\textwidth]{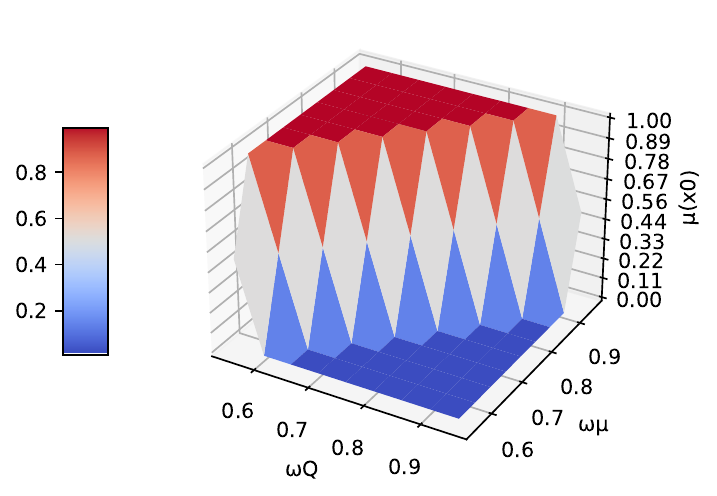}}
\caption{Robustness of the distribution with respect to the different choice of learning rate parameters $\omega^Q$ and $\omega^\mu$.\label{fig:full-omegas}}
\label{fig:A1surface}
\end{figure}
In order to show the stability in the distribution, we also provide the error plot with respect to $\mu(x_0)$. We use the absolute error. 

 \begin{figure}[H]
\center 
\subfloat{\includegraphics[width=0.4\textwidth]{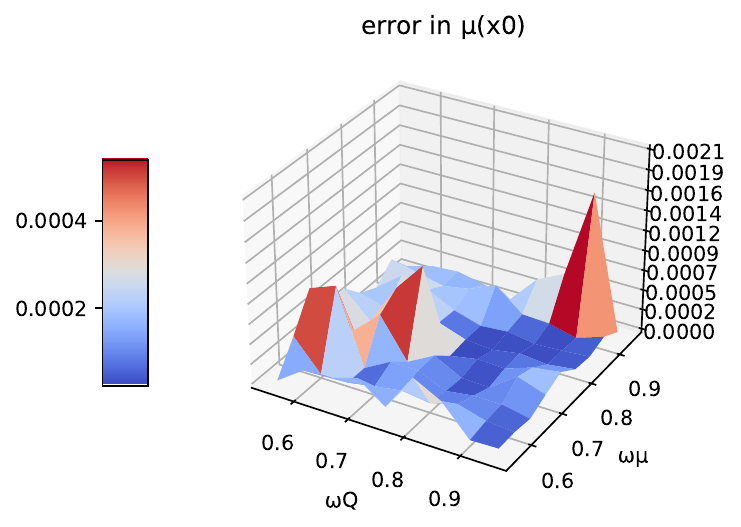}}
\caption{Absolute error in the distribution with respect to the different choice of learning rate parameters $\omega^Q$ and $\omega^\mu$.\label{fig:error-omegas}}
\label{fig:A1error}
\end{figure}
From this plot, one can see that the performance of Algorithm \ref{algo:U2MFQL} and \ref{algo:U2MFQL-MFC} are robust, the error for all combination of learning rates is close to zero.

\section{Conclusion}
\label{sec:conclusion}
In this paper, we  provided precise assumptions ensuring the convergence of the two-timescale RL algorithm proposed in \cite{Andrea20} and we clarified that in the MFC case, one population distribution for each state/action point needs to be updated.  The proof relies on the theory of stochastic approximation. In order to use a two-timescale argument as  in  \cite{Borkar97}, we smooth the strategies using a soft-min function to regularize the dependency on the population distributions.  One of the main properties of the algorithms is that they are characterized  by the ratio of two learning rates leading to the solution of a Mean Field Game problem or of a Mean Field Control problem. To illustrate this key fact, we give an example satisfying the sets of assumptions for the two problems.

\bibliographystyle{apalike}
\bibliography{bibtex}

\appendix 

\section{Lipschitz property of the 2 scale operators}
\label{app3}
We define the $\softmin_\phi$ function in the following way:
\begin{align}\label{softmin}
{\softmin}_\phi(z)=\left(\frac{e^{-\phi z_i}}{\sum_{j}e^{-\phi z_j}}\right)_{i=1,2,...,|\mathcal{A}|}
\end{align}
where $z\in \mathbbm{R}^{|\mathcal{A}|}$.  By \cite{gao2018properties}, $\softmin_\phi$ is Lipschitz continuous for the $L^2$-norm and it is $\phi$-Lipschitz,
\begin{align}
\label{control l}
    \|{\softmin}_\phi(z)-{\softmin}_\phi(z')\|_2\leq \phi\|z-z'\|_2,\hspace{20pt} z,z'\in\mathbbm{R}^{|\mathcal{A}|}
\end{align}
Moreover, since $|\mathcal{A}|$ is finite, all the norms on $\mathbbm{R}^{|\mathcal{A}|}$ are equivalent, so that
\begin{align}\label{softminlip}
    \|{\softmin}_\phi(z)-{\softmin}_\phi(z')\|_1\leq \phi \sqrt{|\mathcal{A}|}\|z-z'\|_1
\end{align}
We also define an $\argmine$ function in the following way: 
\begin{align*}
    \argmine(z)_i=\begin{cases}
        \frac{1}{k}\hspace{1cm} \text{if } z_i = \min z \\
        0\hspace{1cm}\text{otherwise}
    \end{cases}
\end{align*}
where $i=1,2,...,|\mathcal{A}|$. To be more clear, $\argmine$ will assign equal to probability to the minimal value in $z$, and $0$ probability to others. Then we know that as $\phi\to 0$, the $\softmin_\phi$ function is smoother and smoother and the corresponding Lipschitz coefficient $L_s(\phi) \to 0$. On the opposite, the $\softmin_\phi$ will converge to $\argmine$ as $\phi\to \infty$.

To alleviate the notation, we will write $Q(x):=(Q(x,a))_{a\in\mathcal{A}}$ for any $Q\in\mathbbm{R}^{|\mathcal{X}|\times|\mathcal{A}|}$, which implies $Q(x),\softmin_\phi Q(x),\argmine Q(x)\in\mathbbm{R}^{|\mathcal{A}|}$. We also introduce a more general  transition kernel, which can take an input of a probability vector over actions instead of a single action: for $x,x'\in\mathcal{X},\pi\in\Delta^{|\mathcal{A}|},\mu\in\Delta^{|\mathcal{X}|}$,
\begin{align}\label{gp}
    p(x'|x,\pi,\mu) = \sum_a \pi(x)(a)p(x'|x,a,\mu).
\end{align}
We also define:
\begin{align}
    f(x,\pi,\mu) = \sum_a \pi(x)(a)f(x,a,\mu).
\end{align}

\section{Proof of existence and uniqueness of a GASE to the first O.D.E. for $\mu_t$ in \eqref{MFGODE}}\label{GASEproof}
To prove that the GASE exists for the first O.D.E. in \eqref{MFGODE}, we will use a strict contraction property, which requires to control $\phi$ properly. We introduce the additional assumption:

\begin{assumption}\label{controlphic}
   $0<\phi < \frac{|\mathcal{X}|c_{min}-L_p}{|\mathcal{A}|}\frac{1-\gamma}{L_f + \frac{\gamma}{1-\gamma} L_p\|f\|_\infty} =: \phi_{max}$. 
\end{assumption}
With this assumption, we have $(\phi |\mathcal{A}|\frac{L_f + \frac{\gamma}{1-\gamma} L_p\|f\|_\infty}{1-\gamma} + L_p + 1-|\mathcal{X}|c_{min})<1$. We also  have $(\phi|\mathcal{A}|\frac{L_f + \gamma L_p\|Q^*_{\mu}\|_\infty}{1-\gamma} + L_p + 1-|\mathcal{X}|c_{min})<1$, which will give the strict contraction property we need in the next proposition.
\begin{proposition}\label{emu_MFG}
    Under Assumptions~\ref{fp_lipschitz}, \ref{mfclp} and \ref{controlphic}, $\dot{\mu}_t = \mathcal{P}_2(Q^*_{\mu_t},\mu_t)$ has a unique GASE, that we will denote by $\mu^{*\phi}$.
\end{proposition}

\begin{proof}
We show that $\mu \mapsto \mathrm{P}^{\softmin_\phi Q^*_{\mu},\mu}\mu$ is a strict contraction. First, recall that, from Proposition \ref{PLipschitz},  we have 
    \begin{align*}
        \|\mathrm{P}^{\softmin_\phi Q, \mu}\mu-\mathrm{P}^{\softmin_\phi Q, \mu'}\mu'\|_1 &\leq \|\mathrm{P}^{\softmin_\phi Q, \mu}\mu-\mathrm{P}^{\softmin_\phi Q, \mu}\mu')\|_1 + \|\mathrm{P}^{\softmin_\phi Q, \mu}\mu'-\mathrm{P}^{\softmin_\phi Q, \mu'}\mu'\|_1  \\
        &\leq (1-|\mathcal{X}|c_{min}^\phi)\|\mu-\mu'\|_1 + L_p\|\mu-\mu'\|_1 \\
        &\leq (1+L_p-|\mathcal{X}|c_{min})\|\mu-\mu'\|_1.
    \end{align*}
    We have:
    \begin{align*}
        &\|\mathrm{P}^{\softmin_\phi Q^*_{\mu}, \mu}\mu - \mathrm{P}^{\softmin_\phi Q^*_{\mu'}, \mu'}\mu'\|_1 
        \\
        &\leq \|\mathrm{P}^{\softmin_\phi Q^*_{\mu}, \mu}\mu - \mathrm{P}^{\softmin_\phi Q^*_{\mu'}, \mu}\mu\|_1 +\|\mathrm{P}^{\softmin_\phi Q^*_{\mu'}, \mu}\mu - \mathrm{P}^{\softmin_\phi Q^*_{\mu'}, \mu'}\mu'\|_1 \\
        &\leq \|\mathcal{P}_2( Q^*_{\mu},\mu) - \mathcal{P}_2( Q^*_{\mu'},\mu)\|_1 +\|\mathrm{P}^{\softmin_\phi Q^*_{\mu'}, \mu}\mu - \mathrm{P}^{\softmin_\phi Q^*_{\mu'}, \mu'}\mu'\|_1\\
        &\leq \|\mathcal{P}_2( Q^*_{\mu},\mu) - \mathcal{P}_2( Q^*_{\mu'},\mu)\|_1 +  (1+L_p-|\mathcal{X}|c_{min})\|\mu-\mu'\|_1\\
        &\leq \phi|\mathcal{A}|\|Q^*_{\mu}-Q^*_{\mu'}\|_\infty +(1+L_p-|\mathcal{X}|c_{min})\|\mu-\mu'\|_1\\
        &\leq (\phi |\mathcal{A}|\frac{L_f + \gamma L_p\|Q^*_{\mu}\|_\infty}{1-\gamma} + L_p + 1-|\mathcal{X}|c_{min})\|\mu-\mu'\|_1,
    \end{align*}
    which shows the strict contraction property under Assumption \ref{controlphic}. As a result, by contraction mapping theorem \cite{SELL197342}, a unique GASE exists and furthermore by \cite[Theorem 3.1]{563625}, $\mu_t$ will converge to it.
\end{proof}

\section{Proof of existence and uniqueness of a GASE to the second O.D.E. for $Q_t$ in \eqref{MFCODE}}\label{GASEproofMFC}

To prove that the GASE exists for the second O.D.E. in \eqref{MFCODE}, we need to control the parameter $\phi$ properly to have a strict contraction. We will use again Assumption~\ref{controlphic}. 
With this assumption and Assumption \ref{mfclp}, we have $\gamma + (L_f + \gamma L_p\|Q\|_\infty)\frac{\phi|\mathcal{A}|}{|\mathcal{X}|c_{min}-L_p} <1$ which will give the strict contraction property we need in the next proposition.

\begin{proposition}\label{eQ_MFC}
    Suppose Assumption \ref{fp_lipschitz}, \ref{mfclp} and \ref{controlphic} hold. Then $\dot{Q}_t = \mathcal{T}_2(Q_t,{\mu^{*\phi}_{Q_t}})$ has a unique GASE, that we will denote by $Q^{*\phi}$.
\end{proposition}
\begin{proof}
We show that $Q \mapsto \mathcal{B}_{\mu^{*\phi}_Q}Q$ is a strict contraction. We have:
    \begin{align*}
        \|\mathcal{B}_{\mu^{*\phi}_Q}Q - \mathcal{B}_{\mu^{*\phi}_{Q'}}Q'\|_\infty &\leq \|\mathcal{B}_{\mu^{*\phi}_Q}Q - \mathcal{B}_{\mu^{*\phi}_{Q}}Q'\|_\infty+\|\mathcal{B}_{\mu^{*\phi}_Q}Q' - \mathcal{B}_{\mu^{*\phi}_{Q'}}Q'\|_\infty \\
        &\leq \gamma \|Q-Q'\|_\infty + (L_f + \gamma L_p\|Q\|_\infty)\|\mu^{*\phi}_Q - \mu^{*\phi}_{Q'}\|_1 \\
        &\leq \gamma \|Q-Q'\|_\infty + (L_f + \gamma L_p\|Q\|_\infty)\frac{\phi|\mathcal{A}|}{|\mathcal{X}|c_{min}-L_p}\|Q-Q'\|_\infty\\
        &\leq (\gamma + (L_f + \gamma L_p\|Q\|_\infty)\frac{\phi|\mathcal{A}|}{|\mathcal{X}|c_{min}-L_p})\|Q-Q'\|_\infty,
    \end{align*}
    which shows the strict contraction property. As a result, by contraction mapping theorem \cite{SELL197342}, a unique GASE exists and furthermore by \cite[Theorem 3.1]{563625}, it is the limit of $Q_t$.
\end{proof}

\section{MFC Optimal Policy}\label{MFCoptimal}
Recall that we assume $c>2\gamma$ and $p<0.5$. Let us consider a general policy $\pi$ characterized by $ p_{0s}$ and $p_{1s}$ in $[0,1]$:
\[
\pi(s|0) = p_{0s},\quad \pi(s|1) = p_{1s}.
\]

Then, we can compute the limiting/invariant distributions under the policies $\tilde\pi^{(x,a)}$ characterized by the probabilities to be at state $1$:
\[
\mu_1^{(0,s)} = \frac{p}{1-(1-2p)\pi(s|1)}, \quad \mu_1^{(0,m)} =\frac{1-p}{2-2p-(1-2p)\pi(s|1)},
\]
\[
\mu_1^{(1,s)} = \frac{1-p - (1-2p)\pi(s|0)}{1 - (1-2p)\pi(s|0)}, \quad \mu_1^{(1,m)} =\frac{1-p - (1-2p)\pi(s|0)}{2-2p - (1-2p)\pi(s|0)}.
\]

1. First case: if $Q^\pi(0,s)-Q^\pi(0,m)>0$ and $Q^\pi(1,m)-Q^\pi(1,s) > 0$, by the Bellman equation $Q^\pi=\mathcal{B}Q^\pi$ as defined in \eqref{eq:BQ},
we can compute $Q^\pi$:
\[
Q^\pi(0,m) = \frac{c\mu_0^{(0,m)}+\gamma(1-p)(1-c\mu_0^{(0,m)} + c\mu_0^{(1,s)})}{1-\gamma}, \quad Q^\pi(1,s)=Q^\pi(0,m)+1-c\mu_0^{(0,m)} + c\mu_0^{(1,s)},
\]
\[
Q^\pi(0,s) = c\mu_0^{(0,s)}+\gamma\frac{c\mu_0^{(0,m)}+\gamma(1-p)(1-c\mu_0^{(0,m)} + c\mu_0^{(1,s)})}{1-\gamma}+\gamma p(1-c\mu_0^{(0,m)} + c\mu_0^{(1,s)}),
\]
\[
 Q^\pi(1,m)=Q^\pi(0,s)+1-c\mu_0^{(0,s)} + c\mu_0^{(1,m)}.
\]
Then by $Q^*(x,a) := \inf_{\pi} Q^{\pi}(x,a)$, we deduce that $p_{1s} = 1$ and $p_{0s} = 0$ (corresponding to the pure policy $(m,s)$), so that:
\[
Q^*(0,m)=\frac{cp-\gamma p+\gamma}{1-\gamma}, \quad Q^*(1,s)=Q^*(0,m) +1,
\]
\[
Q^*(0,s)=\frac{c}{2}+\gamma\frac{cp-2\gamma p+\gamma +p}{1-\gamma},\quad Q^*(1,m)=Q^*(0,s)+1,
\]
implying that our starting conditions are satisfied:
\begin{align*}
    \begin{cases}
   Q^*(0,s)-Q^*(0,m)=    c\mu_0^{(0,s)} - c\mu_0^{(0,m)} - \gamma (1-2p)(1-c\mu_0^{(0,m)} + c\mu_0^{(1,s)}) =\frac{1}{2}(1-2p)(c-2\gamma)>0,\\
      Q^*(1,m)-Q^*(1,s)=  c\mu_0^{(1,m)} - c\mu_0^{(1,s)} - \gamma (1-2p)(1-c\mu_0^{(0,m)} + c\mu_0^{(1,s)}) =\frac{1}{2}(1-2p)(c-2\gamma)>0.
    \end{cases}
\end{align*}

2. Second case: if $Q^\pi(0,s)-Q^\pi(0,m)<0$ and $Q^\pi(1,m)-Q^\pi(1,s) < 0$,
we can compute $Q^\pi$ ,
\[
Q^\pi(0,s) = \frac{c\mu_0^{(0,s)}+\gamma p (1-c\mu_0^{(0,s)} + c\mu_0^{(1,m)})}{1-\gamma}, \quad Q^\pi(1,m)=Q^\pi(0,s)+1-c\mu_0^{(0,s)} + c\mu_0^{(1,m)},
\]
\[
Q^\pi(0,m) = c\mu_0^{(0,m)}+\gamma\frac{c\mu_0^{(0,m)}+\gamma p (1-c\mu_0^{(0,s)} + c\mu_0^{(1,m)})}{1-\gamma}+\gamma (1-p) (1-c\mu_0^{(0,s)} + c\mu_0^{(1,m)}),
\]
\[
 Q^\pi(1,s)=Q^\pi(0,m)+1-c\mu_0^{(0,m)} + c\mu_0^{(1,s)}.
\]
Then by $Q^*(x,a) := \inf_{\pi} Q^{\pi}(x,a)$, we get that $p_{1s} = 1$ and $p_{0s} = 0$, so that 
\[
Q^*(0,s)=\frac{c/2+\gamma p}{1-\gamma}, \quad Q^*(1,m)=Q^*(0,s) +1,
\]
\[
Q^*(0,m)=cp+\gamma\frac{cp+\gamma p}{1-\gamma} +\gamma(1-p),\quad Q^*(1,s)=Q^*(0,m)+1.
\]
Consequently,
\begin{align*}
    \begin{cases}
        c\mu_0^{(0,m)} - c\mu_0^{(0,s)} + \gamma (1-2p) (1-c\mu_0^{(0,s)} + c\mu_0^{(1,m)})=\frac{1}{2}(1-2p)(2\gamma-c) <0\\
        c\mu_0^{(1,s)} - c\mu_0^{(1,m)} + \gamma (1-2p) (1-c\mu_0^{(0,s)} + c\mu_0^{(1,m)})=\frac{1}{2}(1-2p)(2\gamma-c)< 0
    \end{cases}
\end{align*} 
as $2\gamma<c$ which breaks our assumption. As a result, under this case, we do not have a better policy than $(m,s)$ obtained in the first case..

3. Third Case: if $Q^\pi(0,s)-Q^\pi(0,m)<0$ and $Q^\pi(1,s)-Q^\pi(1,m) < 0$,
we can compute $Q^\pi$ as,
\[
Q^\pi(0,s) = \frac{(1-\gamma+\gamma p)c\mu_0^{(0,s)}+\gamma p(c\mu_0^{(1,s)}+1)}{(1-\gamma)(1-\gamma(1-2p))}, \quad Q^\pi(1,m)=Q^\pi(0,s)-c\mu_0^{(0,s)} + c\mu_0^{(1,m)}+1,
\]
\[
Q^\pi(1,s) = \frac{(1-\gamma+\gamma p)(c\mu_0^{(1,s)}+1)+\gamma pc\mu_0^{(0,s)}}{(1-\gamma)(1-\gamma(1-2p))}, \quad Q^\pi(0,m)=Q^\pi(1,s)-1-c\mu_0^{(1,s)} + c\mu_0^{(0,m)}.
\]
Then by $Q^*(x,a) := \inf_{\pi} Q^{\pi}(x,a)$, we get that $p_{1s} = 1$ and $p_{0s} = 0$, so that 
\[
Q^*(0,s)  = \frac{(1-\gamma+\gamma p)\frac{c}{2}+\gamma p(cp+1)}{(1-\gamma)(1-\gamma(1-2p))}, \quad Q^*(1,m)=\frac{(1-\gamma+\gamma p)\frac{c}{2}+\gamma p(cp+1)}{(1-\gamma)(1-\gamma(1-2p))}+1,
\]
\[
Q^*(1,s)  = \frac{(1-\gamma+\gamma p)cp+\gamma p(\frac{c}{2}+1)}{(1-\gamma)(1-\gamma(1-2p))}, \quad Q^*(0,m)=\frac{(1-\gamma+\gamma p)cp+\gamma p(\frac{c}{2}+1)}{(1-\gamma)(1-\gamma(1-2p))}-1,
\]
and
\[
Q^*(0,m)-Q^*(0,s) = (1-2p)(\gamma-\frac{c}{2})-1 <0
\]
which is  a contradiction. As a result, under this case, we do not have a better policy than $(m,s)$.

Similarly we do not have a better policy than $(m,s)$ for case $Q^\pi(0,s)-Q^\pi(0,m)>0$ and $Q^\pi(1,m)-Q^\pi(1,s) > 0$.

In conclusion, we have shown that the unique solution to the MFC problem is the pure policy $(m,s)$,
so that Assumption \ref{uniqueness minimum} is satisfied.

\end{document}